\documentclass{amsart}

\usepackage{amsmath}
\usepackage{graphicx}
\usepackage{amssymb}
\usepackage{amsfonts}
\usepackage{mathtools}

\swapnumbers
\newtheorem{theorem}{Theorem}[section]
\newtheorem{lemma}[theorem]{Lemma}
\newtheorem{corollary}[theorem]{Corollary}
\newtheorem{proposition}[theorem]{Proposition}
\theoremstyle{definition}
\newtheorem{definition}[theorem]{Definition}

\newtheorem{remark}[theorem]{Remark}

\numberwithin{equation}{section}
 \theoremstyle{plain}
 
 \numberwithin{equation}{section} 
 \numberwithin{figure}{section} 
 \theoremstyle{plain}
 \theoremstyle{remark}
 \newtheorem*{acknowledgement*}{Acknowledgement}

\newcounter{property}
\newenvironment{property}[1][]{\refstepcounter{property}\par\medskip
\noindent\textbf{A\theproperty. #1} \rmfamily}{\medskip}
\newenvironment{property*}[1][]{\par\medskip
\noindent\textbf{A\theproperty#1'. } \rmfamily}{\medskip}

\newcommand{\diam}{{\text{diam}}}

\usepackage{epstopdf}
\usepackage{amsthm}
\usepackage{enumerate}

\usepackage{dsfont}

\providecommand{\floor}[1]{\left\lfloor \, #1 \, \right\rfloor}

\def\supp{\operatorname{supp}}

\newcommand{\tom}{{\tilde{\Omega} }}
\newcommand{\ttau}{{\tilde{\tau} }}
\newcommand{\tmu}{{\tilde{\mu} }}

\newcommand{\tT}{{\tilde{T} }}
\newcommand{\tU}{{\tilde{U} }}
\newcommand{\tV}{{\tilde{V} }}

\newcommand{\cA}{{\mathcal A}}

\newcommand{\cC}{{\mathcal C}}

\newcommand{\cE}{{\mathcal E}}
\newcommand{\cF}{{\mathcal F}}
\newcommand{\cG}{{\mathcal G}}
\newcommand{\cH}{{\mathcal H}}

\newcommand{\cL}{{\mathcal L}}

\newcommand{\cN}{{\mathcal N}}

\newcommand{\cU}{{\mathcal U}}

\newcommand{\Om}{{\Omega}}
\newcommand{\om}{{\omega}}
\newcommand{\ve}{{\varepsilon}}
\newcommand{\del}{{\delta}}
\newcommand{\Del}{{\Delta}}
\newcommand{\gam}{{\gamma}}
\newcommand{\Gam}{{\Gamma}}

\newcommand{\vr}{{\varrho}}
\newcommand{\Sig}{{\Sigma}}
\newcommand{\sig}{{\sigma}}
\newcommand{\al}{{\alpha}}
\newcommand{\be}{{\beta}}
\newcommand{\ka}{{\kappa}}
\newcommand{\la}{{\lambda}}

\newcommand{\up}{{\upsilon}}

\newcommand{\bbN}{{\mathbb N}}

\newcommand{\bbZ}{{\mathbb Z}}
\newcommand{\bbI}{{\mathbb I}}

\newcommand{\bfM}{{\bf M}}

\newcommand{\bfY}{{\bf Y}}
\newcommand{\bfX}{{\bf X}}

\def\Sxy{\Sigma^{x,y}_r}

\begin{document}
\title[]{Geometric law for numbers of returns\\
until a hazard under $\phi$-mixing}%
 \vskip 0.1cm
\author{ Yuri Kifer
\quad\quad and\quad\quad\quad\quad Fan Yang\\
\vskip 0.1cm
Institute of Mathematics\quad\quad\quad Department of Mathematics\\
The Hebrew University\quad\quad\quad University of Oklahoma\\
Jerusalem, Israel\quad\quad\quad\quad\quad\quad Norman, OK, USA}%
\address{
Institute of Mathematics, The Hebrew University, Jerusalem 91904, Israel}
\email{ kifer@math.huji.ac.il}
\address{
Department of Mathematics, University of Oklahoma,
 Norman, OK 73012, USA}%
\email{fan.yang-2@ou.edu}%

\thanks{ }
 \date{\today}
\begin{abstract}\noindent
 We consider a $\phi$-mixing shift $T$ on a sequence space $\Om$ and study
the number of returns $\{ T^k\om\in U\}$ to a union $U$ of cylinders
of length $n$ until the first return $\{ T^k\om\in V\}$ to another union $V$ of
cylinder sets of length $m$. It turns out
that if probabilities of the sets $U$ and $V$ are small and of the same order
then the above number of returns has approximately  geometric distribution.
Under appropriate conditions we extend this result for some dynamical systems
to geometric balls and Young towers with integrable tails.
This work is motivated by a number of papers on asymptotical behavior of numbers
of returns to shrinking sets, as well as by the papers
on open systems studying their behavior until an exit through a "hole".
\end{abstract}
\maketitle
\markboth{Yu.Kifer and F. Yang}{Returns until a hazard}
\renewcommand{\theequation}{\arabic{section}.\arabic{equation}}
\pagenumbering{arabic}

\section{Introduction}\label{sec1}\setcounter{equation}{0}

Let $\xi_0,\,\xi_1,\,\xi_2,...$ be a sequence of independent identically
distributed (i.i.d.) random variables and $\Gam_N,\,\Del_N$ be a sequence
of sets such that $\Gam_N\cap\Del_N=\emptyset$ and
\[
\lim_{N\to\infty}\la^{-1}NP\{\xi_0\in\Gam_N\}=\lim_{N\to\infty}\nu^{-1}N
P\{\xi_0\in\Del_N\}=1.
\]
Then by the classical Poisson limit theorem
\[
S_N^{(\la)}=\sum_{n=0}^{N-1}\bbI_{\Gam_N}(\xi_n)\,\,\mbox{and}\,\,
S_N^{(\nu)}=\sum_{n=0}^{N-1}\bbI_{\Del_N}(\xi_n),
\]
where $\bbI_\Gam$ is the indicator of a set $\Gam$, converge in distribution as
$N\to\infty$ to Poisson random variables with parameters $\la$
and $\nu$, respectively. On the other hand, if
\[
\tau_N=\min\{ n\geq 0:\,\xi_n\in\Gam_N\}
\]
then it turns out that the sum
\[
S_{\tau_N}=\sum_{n=0}^{\tau_N-1}\bbI_{\Del_N}(\xi_n),
\]
which counts returns to $\Del_N$ until the first arrival at $\Gam_N$,
converges in distribution as $N\to\infty$ to a geometric random variable $\zeta$ with the
parameter $p=\la(\la+\nu)^{-1}$, i.e. $P\{\zeta=k\}=(1-p)^kp$.

In \cite{KR1} a more general setup was considered where $\xi_0,\,\xi_1,\,\xi_2,...$
is a $\psi$-mixing stationary process and which included increasing functions
$q_i(n),\, i=1,...,\ell,\, n\geq 0$ taking on integer values on integers
and satisfying $0\leq q_1(n)<q_2(n)<\cdots q_\ell(n)$ with all differences
$q_i(n)-q_{i-1}(n)$ tending to $\infty$ as $n\to\infty$. There
 "nonconventional" sums
\begin{equation}\label{1.1}
S_{\tau_N}=\sum_{n=0}^{\tau_N-1}\prod_{i=1}^\ell\bbI_{\Del_N}(\xi_{q_i(n)})
\end{equation}
were considered with
\begin{equation}\label{1.2}
\tau_N=\min\{ n\geq 0:\,\prod_{i=1}^\ell\bbI_{\Gam_N}(\xi_{q_i(n)})=1\}
\end{equation}
and $\tau_N=\infty$ if the set in braces is empty. Now $S_{\tau_N}$
equals the number $\cN_N$ of multiple returns to $\Del_N$ until the first
multiple return to $\Gam_N$. It turned out that if
\begin{equation}\label{1.3}
\lim_{N\to\infty}\la^{-1}N(P\{\xi_0\in\Gam_N\})^\ell=\lim_{N\to\infty}\nu^{-1}N
(P\{\xi_0\in\Del_N\})^\ell=1
\end{equation}
then, again, $S_{\tau_N}$ converges in distribution to a geometric random
variable with the parameter $p=\la(\la+\nu)^{-1}$.

 We consider in this paper the following setup which comes from dynamical
 systems but has also a perfect probabilistic sense. Let $\zeta_k,\,
 k=0,1,2,...$ be a $\phi$-mixing discrete time process evolving on a finite
 or countable state space $\cA$. For each sequence $a=(a_0,a_1,a_2,...)\in
 \cA^\bbN$ of elements from $\cA$ and any $m\in\bbN$ denote by $a^{(m)}$
 the string $a_0,a_1,...,a_{m-1}$ which determines also an $m$-cylinder
 set $A^a_m$ in $\cA^\bbN$ consisting of sequences whose initial $m$-string
 coincides with $a_0,a_1,...,a_{m-1}$. Let $\tau^a_m$ be the first $l$ such
  that starting at the time $l$ the process $\zeta_k=\zeta_k(\om),\,
 k\geq 0$ repeats the string $a^{(m)}=(a_0,...,a_{m-1})$.
 Let $b=(b_0,b_1,...)\in\cA^\bbN,\, b\ne a$. We are interested
 in the number of $j<\tau^a_m$ such that process $\zeta_k$ repeats the string
  $b^{(n)}=(b_0,...,b_{n-1})$ starting at the time $j$. Employing the left shift
 transformation $T$ on the sequence space $\cA^\bbN$ we can represent the
 number in question as a random variable on $\Om=\cA^\bbN$ given by the sum
 \begin{equation}\label{1.4}
 \Sig_{n,m}^{b,a}(\om)=\sum_{j=0}^{\tau_m^a-1}\bbI_{A^b_n}(T^{j}\om).
 \end{equation}
 We will show that for any $T$-invariant $\phi$-mixing probability measure $P$
 on $\Om$ and $P$-almost all $a,b\in\Om$ the distribution of random variables
 $\Sig^{b,a}_{n,m}$ approaches in the total variation distance as $n\to\infty$
 the geometric distribution with the parameter
 \[
 P(A^a_m)\big ( P(A^a_m)+P(A^b_n)\big )^{-1}
 \]
 provided the ratio $P(A^b_n)/P(A^a_m)$ stays bounded away from zero and
 infinity. In particular, if this ratio tends to $\la$ when $m=m(n)$ and
 $n\to\infty$ then the distribution of $\Sig_{n,m}^{b,a}$ converges
 in total variation distance to the geometric distribution with the
 parameter $(1+\la)^{-1}$. In fact, we will prove such results for arbitrary
 unions of cylinders of the same length.
 It turns out that under just $\phi$-mixing (and not $\psi$-mixing) our method
 does not work for stationary processes as considered in \cite{KR1} and it
 is not applicable for shifts when the nonconventional sums are considered
 as in \cite{KR1}.

 After proving results for $\phi$-mixing shifts on symbolic spaces we consider
  dynamical systems which can be modeled by such shifts via corresponding partitions
  and approximating geometric balls by elements of such partitions we obtain under
  appropriate conditions geometric limit laws for numbers of returns
  to a sequence of shrinking
  balls until first arrival to a ball from another such sequence. We show also that our
   conditions
  are satisfied for Young towers with exponentially small tails obtaining thus limiting
  geometric law for dynamical systems which can be modelled by such towers.

 Our results are applicable to large  classes of dynamical systems.
  Among such systems are smooth expanding endomorphisms
 of compact manifolds and Axiom A (in particular, Anosov)
 diffeomorphisms which have symbolic representations via Markov partitions
 (see \cite{Bo}). Then, in place of cylinder sets we can count
 returns to an element of a Markov partition until first return
 to another element of this partition. If for such dynamical systems we
 consider Sinai-Ruelle-Bowen type measures then the results are extended
 to returns to geometric balls in place of elements of Markov partitions using
 approximations of the former by unions of the latter (cf.
 the proof of Theorem 3 in \cite{HP}). The results remain true for some
 systems having symbolic representations with infinite alphabet, for instance,
 for the Gauss map $Tx=\frac 1x$ (mod 1), \, $x\in(0,1],\, T0=0$ of the unit
 interval considered with the Gauss measure $G(\Gam)=\frac 1{\ln 2}\int_\Gam
 \frac {dx}{1+x}$ which is known to be $T$-invariant and $\psi$-mixing with
 an exponential speed (\cite{He}). More generally, our weaker $\phi$ mixing
 assumption enables us to consider additional classes of dynamical systems
 among them some non uniformly expanding transformations which can be
 modeled by some Young towers (see \cite {Yo1} and \cite {Yo2}) with
 polynomially fast decaying tails, as well as some Markov processes with
 infinite state space.

 The motivation for the present paper is two-fold. On one hand, it comes from
  the series of papers deriving Poisson type asymptotics for distributions of
  numbers of returns to appropriately shrinking sets
  (see, for instance, \cite{AV}, \cite{AS}, \cite{KR} and references there).
  On the other hand, our motivation was influenced by works on open dynamical
  systems which study dynamics of such systems until they exit the phase space
  through a "hole" (see, for instance, \cite{DWY}, \cite{HY18} and references there). In
  our setup the number of returns is studied until a "hazard" which
  is interpreted as certain visit to a set which can be also
  viewed as a "hole".

  The structure of this paper is as follows. In the next section we will
   describe precisely our setups and formulate main results. In Section
   \ref{sec3} we derive auxiliary lemmas and the corollary from the main
    theorem for the symbolic setup proving the latter in Section \ref{sec4}.
    In Sections \ref{sec5}, \ref{sec6} and \ref{sec7} we prove our results for geometric
    balls. In Section \ref{sec8} we show that the geometric law is preserved under suspension, proving our results for Young's towers.

 \section{Preliminaries and main results}\label{sec2}\setcounter{equation}{0}
\subsection{Symbolic setup}\label{sec2.1}

Our setup consists of a finite or countable set $\mathcal{A}$ which is not a singleton,
the sequence space $\Omega=\mathcal{A}^{\mathbb{N}}$,
the $\sigma$-algebra $\mathcal{F}$ on $\Omega$ generated by cylinder
sets, the left shift $T:\Omega\rightarrow\Omega$, and a $T$-invariant
probability measure $P$ on $(\Omega,\mathcal{F})$ which is assumed to be
$\phi$-mixing with respect to the $\sig$-algebras $\cF_{mn},\, n\geq m$ generated
by the cylinder sets $\{\om=(\om_0,\om_1,...)\in\Om:\,\om_i=a_i\,$ for
$m\leq i\leq n\}$ for some $a_m,a_{m+1},...,a_n\in\cA$. Observe also that
$\cF_{mn}=T^{-m}\cF_{0,n-m}$ for $n\geq m$.

Recall, that the $\phi$-dependence (mixing) coefficient between two $\sig$-algebras
$\cG$ and $\cH$ can be written in the form (see \cite{Br}),
\begin{eqnarray}\label{2.1}
&\phi(\cG,\cH)=\sup_{\Gam\in\cG,\Del\in\cH}\big\{\big\vert\frac {P(\Gam\cap\Del)}
{P(\Gam)}-P(\Del)\big\vert,\, P(\Gam)\ne 0\big\}\\
&=\frac 12\sup\{\| E(g|\cG)-E(g)\|_{L^\infty}:\, g\,\,\mbox{is}\,\,
\cH-\mbox{measurable and}\,\, E|g|_{L^\infty}\leq 1\}.\nonumber
\end{eqnarray}
Set also
\begin{equation}\label{2.2}
\phi(n)=\sup_{m\geq 0}\phi(\cF_{0,m},\cF_{m+n,\infty}).
\end{equation}
The probability $P$ is called $\phi$-mixing if $\phi(n)\to 0$ as $n\to\infty$.

We will need also the $\al$-dependence (mixing) coefficient between two $\sig$-algebras
$\cG$ and $\cH$ which can be written in the form (see \cite{Br}),
\begin{eqnarray*}
&\al(\cG,\cH)=\sup_{\Gam\in\cG,\Del\in\cH}\big\{\big\vert P(\Gam\cap\Del)
-P(\Gam)P(\Del)\big\vert\big\}\\
&=\frac 14\sup\{\| E(g|\cG)-E(g)\|_{L^1}:\, g\,\,\mbox{is}\,\,
\cH-\mbox{measurable and}\,\, E|g|_{L^\infty}\leq 1\}.\nonumber
\end{eqnarray*}
Set also
\begin{equation*}
\al(n)=\sup_{m\geq 0}\al(\cF_{0,m},\cF_{m+n,\infty}).
\end{equation*}

 For each word $a=(a_0,a_1,...,a_{n-1})\in
\cA^n$ we will use the notation $[a]=\{\om=(\om_0,\om_1,...):\, \om_i=a_i,\,
 i=0,1,...,n-1\}$ for the corresponding cylinder set. Without loss of generality
 we assume that the probability of each 1-cylinder set is positive, i.e. $P([a])>0$
 for every $a\in\cA$, and since $\cA$ is not a singleton we have also
 $\sup_{a\in\cA}P([a])<1$. Write $\Omega_{P}$ for the support
of $P$, i.e.
\[
\Om_P=\{\om\in\Om:\,P[\om_0,...,\om_n]>0\,\,\mbox{for all}\,\, n\geq 0\}.
\]
 For $n\ge 1$ set $\mathcal{C}_{n}=\{[w]\::\:w\in\mathcal{A}^{n}\}$. Then $\cF_{0,n}$
 consists of $\emptyset$ and all unions of disjoint elements from $\cC_{n+1}$.
 As usual, for $n,m\ge1$ we set $n\vee m=\max\{n,m\}$, $n\wedge m=\min\{n,m\}$.
 Next, for any $U\in\cF_{0,n-1},\, U\ne\emptyset$ and $V\in\cF_{0,m-1},\, V\ne\emptyset$
 define
\[
\pi(U)=\min\{1\le k\le n\::\:U\cap T^{-k}U\ne\emptyset\}
\]
and
\[
\pi(U,V)=\min\{0\le k\le n\wedge m \::\:U\cap T^{-k}V\ne\emptyset
\mbox{ or }V\cap T^{-k}U\ne\emptyset\}\:.
\]
Set
\[
\tau_V(\omega)=\min\{k\ge 1\::\:T^k\omega\in V\}
\]
with $\tau_V(\om)=\infty$ if the event in braces does not occur.
Next, define
\[
\Sig_{U,V}=\sum_{k=0}^{\tau_V-1}\bbI_{U}\circ T^k
\]
and write
\[
\kappa_{U,V}=\min\{\pi(U,V),\pi(U),\pi(V)\}\:.
\]

For any two random variables or random vectors $Y$ and $Z$ of the same
dimension denote by $\cL(Y)$ and $\cL(Z)$ their distribution and by
\[
d_{TV}(\cL(Y),\,\cL(Z))=\sup_G|\cL(Y)(G)-\cL(Z)(G)|
\]
the total variation distance between $\cL(Y)$ and $\cL(Z)$ where the supremum
is taken over all Borel sets. We denote also by Geo$(\rho),\,\rho\in(0,1)$
the geometric distribution with the parameter $\rho$, i.e.
\[
\mbox{Geo}(\rho)\{ k\}=\rho(1-\rho)^k\,\,\mbox{for each}\,\, k\in\bbN=\{ 0,1,
...\}.
\]

\begin{theorem}\label{thm2.1}
For all integers $n,m,M,R\ge1$ and the sets $U\in\cF_{0,n-1},\, U\ne\emptyset$ and
$V\in\cF_{0,m-1},\, V\ne\emptyset$ such that $U\cap V=\emptyset$ we have
\begin{eqnarray}\label{2.5}
&d_{TV}(\mathcal{L}\big(\Sig_{U,V}),Geo\big(\frac{P(V)}{P(V)+P(U)}\big)\big)\\
&\le 2\big((1-P(V))^{M+1}+P(U)+6MR(P(U)+P(V))^2+6M\phi(R-n\vee m)\nonumber\\
&+8MR(P(U)+P(V))\big(P(U_r)+P(V_r)+\phi(\ka_{U,V}+r-n\vee m)\big)\nonumber
\end{eqnarray}
where $U_r=T^rU,\, V_r=T^rV$ and $r$ is an arbitrary integer satisfying $0\leq r<n\wedge m$.
\end{theorem}

Next, for each $\xi\in\Om$ and $n\geq 1$ write $A^\xi_n=[\xi_0...\xi_{n-1}]\in\cC_n$.
For any $\xi,\eta\in\Om$ set $\tau_n^\eta(\om)=\tau_{A^\eta_n}(\om)$ and
$\Sig_{n,m}^{\xi,\eta}=\Sig_{A_n^\om,A_m^\eta}$.
 The following corollary deals with the limit behaviour of $\Sig^{\xi,\eta}_{n,m(n)}$
 for $P\times P$-typical pairs $(\xi,\eta)\in\Omega\times\Omega$, where $|m(n)-n|=o(n)$ and,
  as usual, $o(n)$ denotes an unspecified function $f:\mathbb{N}\rightarrow\mathbb{N}$ with
$\frac {f(n)}{n}\rightarrow0$ as $n\rightarrow\infty$.
\begin{corollary}\label{cor2.2}
Let $\{m(n)\}_{n\ge1}\subset\mathbb{N}\setminus\{0\}$ be a sequence satisfying $|m(n)-n|=o(n)$
 as $n\rightarrow\infty$. Assume that there exists $\beta>0$
 such that $\phi(n)\leq\be^{-1}e^{-\be n}$ for all $n\geq 1$.
 Impose also the finite entropy condition $-\sum_{a\in\cA}P([a])\ln P([a])<\infty$.
  Then for $P\times P$-a.e. $(\xi,\eta)\in\Omega\times\Omega$,
\begin{equation}\label{2.6}
\underset{n\to\infty}{\lim}\:d_{TV}(\mathcal{L}(\Sig^{\xi,\eta}_{n,m(n)}),
Geo(\frac{P(A_{m(n)}^{\eta})}{P(A_{m(n)}^{\eta})+P(A_{n}^{\xi})}))=0\:.
\end{equation}
In particular, if
\begin{equation}\label{2.7}
\lim_{n\to\infty}\frac {P(A_n^\xi)}{P(A^\eta_{m(n)})}=\la
\end{equation}
then $\cL(\Sig_{n,m(n)}^{\xi,\eta})$ converges in total variation as
$n\to\infty$ to the geometric distribution with the parameter
$(1+\la)^{-1}$.
\end{corollary}

We observe that, in general (in fact, "usually"), the ratio
$\frac {P(A_n^\xi)}{P(A^\eta_{n})}$ will be unbounded for distinct
$\xi,\eta\in\Om$, and so in order to obtain nontrivial limiting
geometric distribution it is necessary to choose cylinders $A_n^\xi$
and $A_{m(n)}^\eta$ with appropriate relative lengths. In order to have the ratio
$\frac {P(A_n^\xi)}{P(A^\eta_{m(n)})}$
 bounded away from zero and infinity our condition $|m(n)-n|=o(n)$ is,
 essentially, necessary (at least, in the finite entropy case)
 which follows from the Shannon-McMillan-Breiman theorem (see \cite{Pe}).
 Corollary \ref{cor2.2} can be applied to Markov shifts with infinite state space
 satisfying the Doeblin condition which are known to be $\phi$-mixing with an
 exponential speed (see \cite{Br}).

 \begin{remark}\label{rem2.2+}
 By a slight modification of the proof it is possible to show that Corollary \ref{cor2.2}
 holds true for any nonperiodic $\xi$ and $\eta$ which are not shifts of each other.
  \end{remark}

\begin{remark}\label{rem2.3}
Theorem \ref{thm2.1} and Corollary \ref{cor2.2} remain true also for the two-sided shift setup. In
this case $\Om=\cA^\bbZ$, i.e. $\Om=\{\om=(...,\om_{-1},\om_0,\om_1,...):\,\om_i\in\cA\}$, and the
$\sig$-algebras $\cF_{mn}$ are defined for $-\infty<m\leq n<\infty$ as unions of cylinder sets
$\{\om=(...,\om_{-1},\om_0,\om_1,...):\,\om_i=a_i$ for $m\leq i\leq n\}$. The $\phi$ and $\al$
dependence coefficients are defined by the same formulas as above with $\phi(n)=\sup_{-\infty<m<\infty}
\phi(\cF_{-\infty,m},\cF_{m+n,\infty})$ and $\al(n)=\sup_{-\infty<m<\infty}\al(\cF_{-\infty,m},
\cF_{m+n,\infty})$. Next, for $n,m\geq 1$ we consider nonempty sets $U\in\cF_{-n+1,n-1}$ and $V\in\cF_{-m+1,m-1}$
and define $\pi(U),\,\pi(U,V),\,\tau_V,\,\ka_{UV}$ and $\Sig_{U,V}$ in the same way as above. The quantities
we will have to estimate in the proof of Theorem \ref{thm2.1} and Corollary \ref{cor2.2} have the form
$P(U\cap T^{-k}V)=P(T^{-l}U\cap T^{-l-k}V)$ where $U\in\cF_{-n+1,n-1}$ and $V\in\cF_{-m+1,m-1}$. Choosing
$l\geq\max(m-1,n-1)$ we obtain that $T^{-l}U\in\cF_{l-n+1,l+n-1}$ and $T^{-l-k}V\in\cF_{l+k-m+1,l+k+m-1}$
with $l-n+1\geq 0$ and $l-m+1\geq 0$ which amounts to the same estimates as in the one-sided shift case.
\end{remark}
\subsection{Maps with $\phi$-mixing partitions}\label{sec2.2}

Let $T$ be a measurable map of a compact metric space $\bfM$ and $\mu$ be a $T$-invariant probability
measure on $\bfM$. Our setup includes also a countable (one-sided) measurable generating partition $\cA$
of $\bfM$ with finite entropy and denote
by $\cA^n=\bigvee_{j=0}^{n-1}T^{-j}\cA$ its $n$th join. Recall, that "generating" means that if $A_n(x)$
 is an element of the partition $\cA^n$ which contains a point $x$ then $\cap_nA_n(x)=\{ x\}$. Let
  $A^j,\, j=1,2,...$ be a
numeration of elements of $\cA$ then each $x\in M$ has a symbolic representation $\om(x)=(\om_0,\om_1,...)$
so that $\om_k=j$ if $T^kx\in A^j$. Since $\cA$ is generating then no two different points have the same symbolic
representation. Hence, map $\om:\, \bfM\to\cA^\bbN$ is a bijection and it sends the measure $\mu$ to a probability
measure $\om(\mu)$ on $\cA^\bbN$ invariant under the left shift on $\cA^{\bbN}$ which provides a symbolic
representation of the dynamical system $(\bfM,T,\mu)$. Let $\cF_{0,n-1}$ be the $\sig$-algebra generated by all elements of the partition $\cA^n$. Clearly, $\cF_{0,n-1}$ consists of $\emptyset$ and all unions of elements
of $\cA^n$. We denote also by $\cF$ the minimal $\sig$-algebra which contains all $\cF_{0,n-1},\, n\geq 1$.
Recall. that the measure $\mu$ is called (left) $\phi$-mixing if
\[
|\mu(\Gam\cap T^{-n-k}\Del)-\mu(\Gam)\mu(\Del)|\leq\phi(k)\mu(\Gam)
\]
for any $\Gam\in\cF_{0,n-1}$ and $\Del\in\cF$ where $\phi(k)$ is nonincreasing and $\phi(k)\to 0$
as $k\to\infty$. This definition corresponds to the one given above in the symbolic setup if we introduce
$\sig$-algebras $\cF_{mn}=T^{-m}\cF_{0,n-m+1},\, n\geq m$. We will assume the following properties. Denote
by $B_r(x)$ an open ball of radius $r$ centered at $x$ though in view of Assumption A4 below our results remain the
same whether we consider open or closed balls.

\begin{property}[The diameter of $\cA^n$.]\label{diameter}
There is a generating partition $\cA$ and constants $C,p>0$, such that for every $n\ge 0$, the diameter of its $n$th join under the map $T$, $\cA^n$, satisfies $\diam\cA^n \le Cn^{-p}$.
\end{property}

When the diameter of $\cA^n$ decays super-polynomially, we say that A\ref{diameter} holds with $p=+\infty$

\begin{property}[Polynomial rate of $\phi$-mixing.]\label{mixing}
The measure $\mu$ is left $\phi$-mixing with respect to the partition $\cA$, with $\phi(n) \le C n^{-\beta}$ for some $\beta>1$.
\end{property}

\begin{property}[Dimension for $\mu$.]\label{dimension}
There is $d>0$, such that for almost every $x\in\supp(\mu)$,  we have
$$
\lim_{r\to 0}\frac{\log \mu(B_r(x))}{\log r} = d.
$$

It is shown in~\cite{P97} that when this assumption holds, then the Hausdorff dimension dim$_H\mu$
of $\mu$ equals $d$.
\end{property}

\begin{property}[Regularity of $\mu$.]\label{regularity}
There are $C,a>0$ and $b\in \mathbb{R}$, such that for $\delta \ll r$ and almost every $x$,
$$
\mu(B_{r+\delta}(x)\setminus B_{r-\delta}(x))\le C\, \delta^a r^{-b}\, \mu(B_r(x)).
$$
Note that for the Lebesgue measure on $\mathbb{R}^n$, this property is satisfied with $a=b=1$.
\end{property}

For any $x,y,z \in\bfM$, write
$$
\Sxy(z) = \sum_{j=0}^{\tau_{B_r(y)-1}} \mathbb{I}_{B_r(x)}(T^jz),
$$
which counts the number of arrivals to $B_r(x)$ before hitting $B_r(y)$ for the first time.
With these we can state the theorem.

To simplify notation, we will write $\rho(x,y,r) = \frac{\mu(B_r(y))}{\mu(B_r(x))+\mu(B_r(y))}$.
\begin{theorem}\label{t.phimixing}
Assume that Assumptions A1--A4 are satisfied with
\begin{equation}\label{c.parameters}
p>\frac{d+b}{ad}.
\end{equation}
Then for $\mu \times \mu$ almost every $(x,y)\in\bfM\times\bfM$, such that $\rho(x,y,r)\to \rho(x,y)\in(0,1)$ as
$r\to 0$,
$$
\lim_{r\to 0}d_{TV}(\cL(\Sxy),Geo(\rho(x,y))) = 0.
$$
If the diameter of $\cA^n$ decreases super-polynomially then the assumption (\ref{c.parameters}) is
redundant.
\end{theorem}
\begin{remark}
When the measure $\mu$ is absolutely continuous with respect to the volume with a density $h$ that is bounded from above, we can take $a=b=1$ in A\ref{regularity}. In this case, condition~(\ref{c.parameters}) reduces to $p>\frac{d+1}{d}$.
\end{remark}

\begin{remark}[Radius of the ball]
In Theorem~\ref{t.phimixing} we take the ball at $x$ and $y$ with the same radius. However, one can easily check that the same results hold as long as $\frac{\mu(B_{r_y}(y))}{\mu(B_{r_x}(x))+\mu(B_{r_y}(y))}$ converges to a limit $\rho\in(0,1)$ when $r_x$ and $r_y$ tend to $0$. For example, when the measure $\mu$ is absolutely continuous with respect to the Lebesgue measure, one can take the balls to be  $B_{r}(x)$ and $B_{cr}(y)$) for any constant $c>0$.
\end{remark}

\begin{remark}[Invertible case]
Theorem~\ref{t.phimixing} remains true when $T$ is invertible. In this case, assumption A\ref{diameter} should be stated for the two-sided join
$\cA^n = \bigvee_{j-(n-1)}^{n-1}T^{-j}\cA$. The rest of the proof remains the same, with minor modification described in Remark~\ref{rem2.3}. Also note that Theorem~\ref{6.1} and Propositions~\ref{6.5}, \ref{6.6} hold
true for both invertible and non-invertible systems.
\end{remark}

\subsection{Geometric law under suspension}

Next, we will state a general result which will allow us to generalize the previous theorems to discrete time suspensions over  $\phi$-mixing systems.
 Namely, let $(\tom, \tmu, \tT)$ be a measure preserving dynamical system with $\tmu$ being a probability measure. Given a measurable function $R: \tom \to \bbZ^+$ consider the space $\Om = \tom\times \bbZ^+ /\sim$ with the equivalence relation $\sim$ given by
$$
(x,R(x))\sim (\tT(x), 0).
$$
Define {\em the (discrete-time) suspension map over $\tom$ with roof function $R$} as the measurable map $T$ on the space $\Om$ acting by
$$\
T(x,j)=\left\{\begin{array}{ll}(x,j+1) &\mbox{if } j<R(x)-1,\\
	(\hat{T}x,0)&\mbox{if } j=R(x)-1.\end{array}\right.
$$
We will call $\Om$ a {\em tower over $\tom$} and refer to the set $\Omega_k:=\{(x,k): x\in\tom, k<R(x)\}$ as the {\em $k$th floor} where
 $\tom$ can be naturally identified with the $0$th floor called {\em the base of the tower}.

For $0\le k < i$, set $\Om_{k,i} = \{(x,k): R(x) = i\}$. The map
$$
\Pi: (x,k)\mapsto x
$$
is naturally viewed as a projection from the tower $\Om$ to the base $\tom$ and for any given set $U\subset\Om$ we will write
$$
\tU = \Pi(U).
$$

The measure $\tmu$ can be lifted to a measure $\hat\mu$ on $\Om$ by
$$
\hat\mu(A) = \sum_{i=1}^{\infty}\sum_{k=0}^{i-1} \tmu(\Pi(A\cap \Om_{k,i})) = (\tmu\times \cN)(A),
$$
where $\cN$ is the counting measure on $\bbN$. It is easy to verify that $\hat\mu$ is $T$-invariant and if $\tmu(R) = \int R\, d\tmu<\infty$ then $\hat\mu$ is a finite measure. In this case, the measure
$$
\mu = \frac{\hat\mu}{\tmu(R)} = \frac{\tmu\times \cN}{\tmu(R)}
$$
is a $T$-invariant probability measure on $\Om$.
In order to state our main theorem for this section, we will introduce the following class of sets.

\begin{definition}\label{d.1}
We say that a positive measure set $U\subset \Om$ is {\em well-placed} (WP) if for $\tmu$ almost every $x$ the intersection $U\cap\{(x,k):k<R(x)\}$
contains at most one point.
\end{definition}

\begin{remark}\label{r.WP}
	If $U$ is well-placed, then, clearly, $\Pi\mid_U$ is almost injective, i.e. there exists a set $U_0\subset U$ with $\mu(U\setminus U_0)=0$, such that $\Pi\mid_{U_0}$ is injective. In particular, if $U\subset \Om_0$ then it is well-placed. Also note that if $U$ is well-placed, then so is every positive measure subset of $U$.
\end{remark}

As before, denote by $\tau_U$ the first hitting time of $U\subset\Om$, i.e. $\tau_U(x)=\min\{ l\geq 1:\, T^lx\in U\}$ if the event in braces occurs
and $\tau_U(x)=\infty$, if not, and for $U,V\subset\Om$ we also write
$$
\Sig_{U,V}=\sum_{k=0}^{\tau_V-1}\bbI_{U}\circ T^k.
$$
Similarly, given $\tilde{U},\tilde V\subset\tom$, we denote by $\ttau_{\tilde{U}}$ the first hitting time of $\tilde{U}$ under iterates of the map
$\tT$ and set
$$
\tilde\Sig_{\tU,\tV}=\sum_{k=0}^{\ttau_V-1}\bbI_{\tilde U}\circ \tT^k.
$$
In other words, every term that contains tilde is defined for the base systems $(\tom, \tmu, \tT)$, and every term without tilde is defined for the suspension $(\Om,\mu,T)$.

\begin{theorem}\label{t.suspension}
	Let $(\Om,\mu,T)$ be a discrete-time suspension over an ergodic system $(\tom, \tmu, \tT)$ with a roof function $R$ satisfying $\int R\,d\tmu<\infty$. 
Let $\{U_n\}$, $\{V_n\}$ be two sequences of well-placed subsets of $\Om$. We consider the following two cases:
	
	(i) If the base system $(\tom, \tmu,\tT)$ and the sets $\tU_n = A^\om_n, \tV_n = A^\eta_{m(n)}$ fall within the symbolic setup of Section \ref{sec2.1}
 and satisfy the assumptions of Corollary~\ref{cor2.2}, then $\Sig_{U_n,V_n}$ converges in distribution on $(\Om,\mu)$ to the geometric distribution $Geo((1+\la)^{-1})$ with $\la$ given by~\eqref{2.7};
	
	(ii) If the base system $(\tom, \tmu,\tT)$ and the sets  $\tU_n = B_{r_n}(x), \tV_n = B_{r_n}(y)$ fall within the setup of Section \ref{sec2.2} and
satisfy the assumptions of Theorem~\ref{t.phimixing} for some sequence of positive real numbers $\{r_n\}$ with $r_n\to 0$, then $\Sig_{U_n,V_n}$ converges 
in distribution on $(\Om,\mu)$ to the geometric distribution $Geo(\rho(x,y))$.
\end{theorem}

\section{Some auxiliary lemmas and Corollary \ref{cor2.2}}\label{sec3}\setcounter{equation}{0}

We start with the following result which appears in \cite{GS} and in \cite{Ab}
as Lemma 1 but under the extra condition that $\cA$ is finite which is redundant
as the following proof shows.
\begin{lemma}\label{lem3.1}
Suppose that $P$ is $\phi$-mixing then there exists a constant $\up>0$ such that
for any $A\in\cC_n$,
\[
P(A)\leq e^{-\up n}.
\]
\end{lemma}
\begin{proof} Since $\gam=\sup_{a\in\cA}P([a])<1$ and $P$ is $\phi$-mixing, i.e. $\phi(n)\downarrow 0$ as $n\uparrow\infty$, we can set $k=\min\{ j:\,\gam+\phi(j)<1\}$. Let $A=[a_0,a_1,...,a_{n-1}]$. Then
\[
A\subset\bigcap_{i=0}^{[n/k]}T^{-ik}[a_{ik}].
\]
By the definition of the $\phi$-dependence coefficient,
\begin{eqnarray*}
&P(A)\leq P(\bigcap_{i=0}^{[n/k]}T^{-ik}[a_{ik}])\leq (P([a_{k[n/k]}])+\phi(k))P(\bigcap_{i=0}^{[n/k]-1}T^{-ik}[a_{ik}])\\
&\leq\cdots\leq(\gam+\phi(k))^{[n/k]}\gam\leq e^{-\up n}
\end{eqnarray*}
where $\up=-k^{-1}\ln(\gam+\phi(k))$ and, without causing a confusion, we use $[\cdot]$ both for the
integral part of a number and to denote 1-cylinders $[a],\, a\in\cA$.
\end{proof}

Next, we prove Corollary \ref{cor2.2} assuming that Theorem \ref{thm2.1} is already proved but,
first, we will need the following lemma.
In what follows, $\{m(n)\}_{n\ge1}$ is a sequence of positive integers with $m(n)\geq 1$ and
$|m(n)-n|=o(n)$ as $n\rightarrow\infty$. For $n\ge1$ we write $\ka^{\om,\eta}_{n,m}=\ka_{A^\om_n,A^\eta_m}$
and $b(n)=n\wedge m(n)$.
\begin{lemma}\label{lem3.2}Set $c=3\upsilon^{-1}$ and let $\mathcal{E}$
be the set of all $(\omega,\eta)\in\Omega\times\Omega$ for which
there exists $N=N(\omega,\eta)\ge1$ such that $\kappa^{\omega,\eta}_{n,m(n)}
\ge b(n)-c\ln b(n)$ for all $n\ge N$, then $P\times P(\Omega^{2}\setminus
\mathcal{E})=0$.
\end{lemma}

\begin{proof}
For $\omega\in\Omega$ and $n\ge1$ set
\[
B_{\omega,n}=\{\eta\in\Omega\::\:\pi(A_{n}^{\omega},A_{m(n)}^{\eta})\le b(n)-
c\ln b(n)\}.
\]
Assume $b(n)-c\ln b(n)\ge1$. Set $d=d(n)=[b(n)-c\ln b(n)]$ and assume that $d\geq 1$.
Then
\[
P(B_{\omega,n})\le\sum_{r=0}^{d}P\{\eta\::\:
T^{-r}A_{n}^{\omega}\cap A_{m(n)}^{\eta}\ne\emptyset\}+\sum_{r=0}^{d}P
\{\eta\::\:T^{-r}A_{m(n)}^{\eta}\cap A_{n}^{\omega}\ne\emptyset\}.
\]
For $0\le r\le d$,
\[
\{\eta\::\:T^{-r}A_{n}^{\omega}\cap A_{m(n)}^{\eta}\ne\emptyset\}
=T^{-r}[\omega_{0},...,\omega_{n\wedge (m(n)-r)-1}]
\]
and
\[
\{\eta\::\:T^{-r}A_{m(n)}^{\eta}\cap A_{n}^{\omega}\ne\emptyset\}
=[\omega_{r},...,\omega_{n\wedge(m(n)+r)-1}].
\]
Hence by Lemma \ref{lem3.1} for all $n$ large enough,
\begin{eqnarray*}
P(B_{\omega,n}) & \le &
\sum_{r=0}^{d}e^{-\upsilon(n\wedge (m(n)-r))}
+\sum_{r=0}^{d}e^{-\upsilon((n-r)\wedge m(n))}\\& \le &
2\sum_{r=0}^{d}e^{-\upsilon(b(n)-r)}\le
2\frac{e^{-\upsilon(b(n)-d)}}{1-e^{-\upsilon}}\le
\frac{2b(n)^{-3}}{1-e^{-\upsilon}}.
\end{eqnarray*}
From this and since $|b(n)-n|=o(n)$ it follows that $\sum_{n=1}^{\infty}
P(B_{\omega,n})<\infty$,
and so by the Borel-Cantelli lemma
\[
P\{\eta\::\:\#\{n\ge1\::\:\eta\in B_{\omega,n}\}=\infty\}=0\:.
\]
From Fubini's theorem we now get,
\begin{multline*}
P\times P\{(\omega,\eta)\::\:\#\{n\ge1\::\:\pi(A_{n}^{\omega},
A_{m(n)}^{\eta})\le b(n)-c\ln b(n)\}=\infty\}\\
=\int_{\Omega}P\{\eta\::\:\#\{n\ge1\::\:\eta\in B_{\omega,n}\}=
\infty\}\:dP(\omega)=0\:.
\end{multline*}
In a similar manner it can be shown that
\[
P\{\omega\::\:\#\{n\ge1\::\:\pi(A_{n}^{\omega})\le b(n)-c\ln b(n)\}=\infty\}
=0\:
\]
and
\[
P\{\eta\::\:\#\{n\ge1\::\:\pi(A_{m(n)}^{\eta})\le b(n)-c\ln b(n)\}=\infty\}
=0\:
.\]
This completes the proof of the lemma.
\end{proof}
\begin{proof}[Proof of Corollary \ref{cor2.2}]
Let $\kappa^{\omega,\eta}_{n,m(n)}\ge b(n)-c\ln b(n)$ for all $n\ge N$ where
 $N=N(\omega,\eta)\ge1$.

Let $c$ and $\mathcal{E}$ be as in the statement of Lemma \ref{lem3.2}.
Denote by $h$ the entropy of the system $(\Omega,P,T)$ which is finite under our
 assumptions. Let $\mathcal{E}_{0}$ be the set of all $(\omega,\eta)\in\mathcal{E}
\cap(\Omega_{P}\times\Omega_{P})$
for which
\[
-\underset{n\to\infty}{\lim}\frac{\log P(A_{n}^{\omega})}{n}=
-\underset{n\to\infty}{\lim}
\frac{\log P(A_{n}^{\eta})}{n}=h\:.
\]
Recall (see, for instance, \cite{Br}) that our $\phi$-mixing (and even
$\al$-mixing) assumption implies that the shift $T$ is mixing in the
ergodic theory sense with respect to the invariant measure $P$, and so it
is ergodic. Hence, we can apply the Shannon-McMillan-Breiman Theorem
(see, for instance, \cite{Pe}) which implies that the above
 equalities hold true with probability one and $h\geq\up>0$ by
Lemma \ref{lem3.1}. This together with  Lemma \ref{lem3.2}
yields that $P\times P(\Omega^{2}\setminus\mathcal{E}_{0})=0$. Now, let
$(\om,\eta)\in\cE_0$. Taking in Theorem \ref{thm2.1}
$M=M(n)=[e^{(h+\ve)n}]$ with small enough $\ve>0$, $R=R(n)=n^2$ and $r=[n/2]$
 we obtain (\ref{2.6}) from (\ref{2.5}) while assuming (\ref{2.7}) the second assertion
  follows directly from (\ref{2.6}).
\end{proof}

For the proof of Theorem \ref{thm2.1} we will need the following general result.
\begin{lemma}\label{lem3.3} Let $\cF_1,\,\cF_2$ and $\cF_3$ be sub $\sig$-algebras of $\cF$
such that
\begin{equation}\label{1}
\max(\phi(\cF_1,\sig(\cF_2,\cF_3)),\,\phi(\cF_2,\cF_3))=\del
\end{equation}
where $\sig(\cG,\tilde\cG)$ denotes the minimal $\sig$-algebra containing $\cG$ and $\tilde\cG$.
Then
\begin{equation}\label{2}
\al(\cF_2,\,\sig(\cF_1,\cF_3))\leq 3\del.
\end{equation}
\end{lemma}
\begin{proof} Let $U_1\in\cF_1,\, U_2\in\cF_3$ and $V\in\cF_2$. Then
\begin{eqnarray*}
&|P(U_1\cap V\cap U_2)-P(U_1)P(V\cap U_2)|\leq\del P(U_1),\\
& P(U_1)|P(V\cap U_2)-P(V)P(U_2)|\leq\del P(U_1)P(V)\\
&\mbox{and}\,\,\, P(V)|P(U_1)P(U_2)-P(U_1\cap U_2)|\leq\del P(U_1)P(V).
\end{eqnarray*}
Hence
\begin{equation}\label{3}
|P(U_1\cap V\cap U_2)-P(U_1\cap U_2)P(V)|\leq\del P(U_1)(1+2P(V))\leq 3\del P(U_1).
\end{equation}

Next, consider the collection $\cU$ of all sets of the form $U=\bigcup_{i=1}^l(U_1^{(i)}\cap U_2^{(i)})$
where $U_1^{(i)}\in\cF_1,\, i=1,...,l$ are disjoint while $U^{(i)}_2\in\cF_3,\, i=1,...,l$
are arbitrary. Clearly, finite unions and intersections of sets from $\cU$ belong to $\cU$, and so
$\cU$ is an algebra of sets as $\Om$ and $\emptyset$ belong to $\cU$, as well. Now, if
$U=\bigcup_{i=1}^l(U_1^{(i)}\cap U_2^{(i)})$ and $V$ are as above then by (\ref{3}),
\begin{eqnarray}\label{4}
&|P(U\cap V)-P(U)P(V)|=|\sum_{i=1}^lP(U_1^{(i)}\cap V\cap U_2^{(i)})\\
&-P(V)\sum_{i=1}^lP(U_1^{(i)}\cap U_2^{(i)}|\leq\sum_{i=1}^l|P(U_1^{(i)}\cap V\cap U_2^{(i)})\nonumber\\
&-P(V)P(U_1^{(i)}\cap U_2^{(i)}|\leq 3\del P(\bigcup_{i=1}^lU_1^{(i)})\leq 3\del.\nonumber
\end{eqnarray}
The estimate (\ref{4}) being true for all $U_j\in\cU$ remains valid under monotone limits
$U_j\uparrow$ and $U_j\downarrow$, and so it holds true for any $U\in\sig(\cF_1,\cF_3)$ and
$V\in\cF_2$, yielding (\ref{2}).
\end{proof}

We will need the following result which is, essentially, an exercise in
elementary probability whose proof can be found in \cite{KR1}.
\begin{lemma}\label{lem3.4}
Let $Y=\big\{ Y_{k,l}:\, k\geq 0$ and $l\in\{ 0,1\}\big\}$ be independent
Bernoulli random variables such that
$1>P\{ Y_{k,0}=1\}=p=1-P\{ Y_{k,0}=0\}>0$ and $1>P\{ Y_{k,1}=1\}=q=1-
P\{ Y_{k,1}=0\}>0$.
Set $\tau=\min\{ l\geq 0:\, Y_{l,0}=1\}$. Then $S=\sum_{l=0}^{\tau-1}Y_{l,1}$
is a geometric random variable with the parameter $p(p+q-pq)^{-1}$.
\end{lemma}

\section{Proof of Theorem \ref{thm2.1}}\label{sec4}\setcounter{equation}{0}

 Define random variables $X_{k,0}$ and $X_{k,1}$ on $(\Omega,\mathcal{F},P)$ by
\[
X_{k,0}=\bbI_{V}\circ T^k\text{ and }X_{k,1}=\bbI_{U}\circ T^k,
\]
and set
\[
S_M=\sum_{k=0}^{M-1}X_{k,1},\,\,\tau=\min\{ k\geq 0:\, X_{k,0}=1\}\,\,\mbox{and}\,\,
\tau_M=\min(\tau,M).
\]
Then $S_\tau=\Sig_{V,M}$ which appears in Theorem \ref{thm2.1}.
Let $\{ Y_{k,\al}\::\: k\geq 0,\,\al=0,1\}$ be a sequence of independent Bernoulli random variables
 such that $Y_{k,\al}$ has the same distribution as $X_{k,\al}$. Set
\[
S^*_M=\sum_{k=0}^{M-1}Y_{k,1},\,\,\tau^*=\min\{ k\geq 0:\, Y_{k,0}=1\}\,\,\mbox{and}\,\,
\tau^*_M=\min(\tau^*,M).
\]
Denote $p_V=P\{ X_{k,0}=1\}=P\{ Y_{k,0}=1\}$ and $p_U=P\{ X_{k,1}=1\}=P\{ Y_{k,1}=1\}$,
$k=0,1,...$. Observe that $S^*_{\tau^*}$ has by Lemma \ref{lem3.4} the geometric distribution
with the parameter
\[
\vr=\frac {p_V}{p_V+p_U(1-p_V)}>\rho=\frac {p_V}{p_V+p_U}.
\]
Next, we can write
\begin{equation}\label{a0}
d_{TV}(\cL(S_\tau),\,\mbox{Geo}(\rho))\leq A_1+A_2+A_3+A_4
\end{equation}
where $A_1=d_{TV}(\cL(S_\tau),\cL(S_{\tau_M}))$,  $A_2=d_{TV}(\cL(S_{\tau_M}),\cL(S^*_{\tau^*_M}))$,
$A_3=d_{TV}(\cL(S^*_{\tau^*_M}),\cL(S^*_{\tau^*}))$ and $A_4=d_{TV}(\mbox{Geo}(\vr),\mbox{Geo}(\rho))$.

Introduce random vectors $\bfX_{M,\al}=\{ X_{k,\al},\, 0\leq k\leq M\}$, $\al=0,1$, $\bfX_{M}=\{\bfX_{M,0},
\bfX_{M,1}\}$, $\bfY_{M,\al}=\{ Y_{k,\al},\, 0\leq k\leq M\}$, $\al=0,1$ and $\bfY_{M}=\{\bfY_{M,0},
\bfY_{M,1}\}$. Observe that the events $\{ S_\tau\ne S_{\tau_M}\}$ or $\{ S^*_{\tau^*}\ne S^*_{\tau_M^*}\}$
can occur only if $\{\tau>M\}$  or $\{\tau^*>M\}$, respectively. Also, we can write $\{\tau>M\}=\{ X_{k,0}=0$
for all $k=0,1,...,M\}$ and $\{\tau^*>M\}=\{ Y_{k,0}=0$ for all $n=0,1,...,M\}$. Hence,
 \begin{eqnarray}\label{a1}
 &A_1\leq P\{\tau>M\}\leq P\{\tau^*>M\}+|P\{ X_{k,0}=0\,\,\mbox{for}\,\, k=0,1,...,M\}\\
 &-P\{ Y_{k,0}=0\,\,\mbox{for}\,\, k=0,1,...,M\}|\leq P\{\tau^*>M\}\nonumber\\
 &+d_{TV}(\cL(\bfX_{M,0}),\cL(\bfY_{M,0}))\leq (1-p_V)^{M+1}+d_{TV}(\cL(\bfX_M),\,\cL(\bfY_M)).\nonumber
 \end{eqnarray}
 Also,
 \begin{equation}\label{a2}
 A_3\leq P\{ \tau^*>M\}=(1-p_V)^{M+1}.
 \end{equation}
 The estimate of $A_4$ is also easy
 \begin{eqnarray}\label{a3}
 &A_4\leq \sum_{k=0}^\infty |\vr(1-\vr)^k-\rho(1-\rho)^k|\leq 2\sum_{k=1}^\infty((1-\rho)^k-(1-\vr)^k)\\
 &=2(1-\rho)\rho^{-1}-2(1-\vr)\vr^{-1}=\frac {2(\vr-\rho)}{\rho\vr}=2p_U.\nonumber
 \end{eqnarray}

Next, clearly,
\begin{equation}\label{a4}
A_2\leq d_{TV}(\cL(\bfX_M),\cL(\bfY_M))
\end{equation}
and it remains to estimate the right hand side of (\ref{a4}). By Theorem 3 in \cite{AGG},
\begin{equation}\label{a5}
d_{TV}(\cL(\bfX_M),\cL(\bfY_M))\leq 2b_1+2b_2+b_3+2\sum_{0\leq n\leq M,\al=0,1}q^2_{n,\al},
\end{equation}
(the additional factor 2 in \cite{AGG} is due to a different definition of $d_{TV}$), where
$q_{n,0}=p_V$ and $q_{n,1}=p_U$. In order to define $b_1,b_2$ and $b_3$ set
\begin{eqnarray*}
&B^{M,R}_{k,\al}=\{(l,\al),(l,1-\al):\, 0\leq l\leq M,\, |l-k|\leq R\}\,\,\mbox{and}\\
& I_M=\{(k,\al):\, 0\leq k\leq M,\,\al=0,1\}.
\end{eqnarray*}
Then
\[
b_1=\sum_{(k,\al)\in I_M}\sum_{(l,\be)\in B^{M,R}_{k,\al}}q_{k,\al}q_{l,\be},
\]
\begin{equation}\label{b2}
b_2=\sum_{(k,\al)\in I_M}\sum_{(k,\al)\ne(l,\be)\in B^{M,R}_{k,\al}}q_{(k,\al),(l,\be)},
\end{equation}
where $q_{(k,\al),(l,\be)}=E(X_{k,\al}X_{l,\be})$, and
\[
b_3=\sum_{(k,\al)\in I_M}s_{k,\al},
\]
where
\[
s_{k,\al}=E\big\vert E\big(X_{k,\al}-q_{k,\al}|\sig\{ X_{l,\be}:\,(l,\be)\in I_M\setminus
 B^{M,R}_{k,\al}\}\big)\big\vert.
 \]

Clearly,
\begin{equation}\label{a6}
b_1\leq 2MR(p^2_V+2p_V p_U +p^2_U)=2MR(p_V+p_U)^2.
\end{equation}
In order to estimate $b_2$ consider two nonempty sets $D\in\cF_{0,n-1}$ and $E\in\cF_{0,m-1}$.
Then both sets are finite or countable unions of corresponding cylinder sets
\[
D=\cup_i[d_0^{(i)},d_1^{(i)},...,d_{n-1}^{(i)}]\quad\mbox{and}\quad
E=\cup_j[e_0^{(j)},e_1^{(j)},...,e_{m-1}^{(j)}].
\]
Assume that $D\cap T^{-k}E\ne\emptyset$ and suppose that $k+r\geq n$ where an integer $r$
satisfies $0\leq r<m$. Set $E_r=T^rE$, then $T^{-r}E_r\supset E$, and so
$D\cap T^{-k}E\subset D\cap T^{-(k+r)}E_r$. Since $D\in\cF_{0,n-1}$ and
 $T^{-(k+r)}E_r\in\cF_{k+r,\infty}$ we obtain  by the definition of the $\phi$-dependence
coefficient that
\begin{equation}\label{a7}
P(D\cap T^{-k}E)\leq P(D\cap T^{-(k+r)}E_r)\leq P(D)(P(E_r)+\phi(k+r-n+1)).
\end{equation}
If $k\geq l$ and $k-l<\pi(U)$ then $X_{k,1}X_{l,1}=0$, if $k\geq l$ and $k-l<\pi(V)$
then $X_{k,0}X_{l,0}=0$, if $k\geq l$ and $k-l<\pi(U,V)$ then $X_{k,1}X_{l,0}=0$ and,
similarly, if $k\geq l$ and $k-l<\pi(V,U)$ then $X_{k,0}X_{l,1}=0$. Thus, if
$X_{k,\al}X_{l,\be}\ne 0$ and $k\geq l$ then we must have $k-l\geq\ka_{U,V}$. Hence, by
(\ref{a7}), if $k-l\geq\pi(U)$ and $\al=\be=1$ then
\[
EX_{k,\al}X_{l,\be}=P(U\cap T^{-(k-l)}U)\leq p_U(P(U_r)+\phi(\pi(U)+r-n+1))
\]
where $U_r=T^rU$ and $0\leq r<n$. If $k-l\geq\pi(V)$ and $\al=\be=0$ then
\[
EX_{k,\al}X_{l,\be}\leq p_V(P(V_r)+\phi(\pi(V)+r-n+1))
\]
where $V_r=T^rV$ and $0\leq r<m$. If $k-l\geq\pi(U,V)$ and $\al=0,\be=1$ then
\[
EX_{k,\al}X_{l,\be}=P(V\cap T^{-(k-l)}U)\leq p_V(P(U_r)+\phi(\pi(U,V)+r-m+1)).
\]
Finally, if $k-l\geq\pi(U,V)$ and $\al=1,\be=0$ then
\[
EX_{k,\al}X_{l,\be}\leq p_U(P(V_r)+\phi(\pi(U,V)+r-n+1)).
\]
It follows that
\begin{equation}\label{a8}
b_2\leq 4MR(p_U+p_V)(P(U_r)+P(V_r)+\phi(\ka_{U,V}+r-m\vee n)).
\end{equation}

In order to estimate $b_3$ we will use Lemma \ref{lem3.3} which gives
\[
s_{k,\gam}\leq\al\big(\cF_{k,k+n},\sig(\cF_{0,k-R+n\vee m},\cF_{k+R-n\vee m,\infty})\big)
\leq 3\phi(R-n\vee m),\,\,\gam=0,1.
\]
It follows that
\begin{equation}\label{a9}
b_3\leq 6M\phi(R-n\vee m).
\end{equation}
Finally, from (\ref{a0})--(\ref{a6}), (\ref{a8}) and (\ref{a9}) we obtain that
\begin{eqnarray*}
&d_{TV}(\cL(S_\tau),\mbox{Geo}(\rho))\leq 2(1-p_V)^{M+1} +2p_U+8MR(p_V+p_U)^2\\
&+16MR(p_U+p_V)\phi(\ka_{U,V}-|m-n|)+12M\phi(R-n\vee m)+4M(p^2_U+p^2_V)
\end{eqnarray*}
and (\ref{2.5}) follows.
\qed

\begin{remark} Unlike \cite{KR1} where $\psi$-mixing setups were considered, here under only
$\phi$-mixing we cannot extend the result either to stationary processes case or to a nonconventional
 setup. For the former the problem arises in estimates for $b_2$ (close correlations), since above, close
 correlations $EX_{k,\al}X_{l,\be}$ with $|k-l|<\pi(U,V)$ are zero while in the stationary process case we do
 not have sufficient control of close correlations under just $\phi$-mixing.
 In a nonconventional setup we would need a result of the type Lemma \ref{lem3.3} with more than 3
 $\sig$-algebras similar to Lemma 3.3 in \cite{KR} which was proved under $\psi$-mixing. Such a result
 is not available under $\phi$-mixing.
 \end{remark}

\section{Estimates on balls approximations}\label{sec5}\setcounter{equation}{0}
Throughout this section, $\mu$ is a $T$ invariant ergodic measure, and $\cA$ is a measurable, generating partition of $\bfM$ satisfying Assumption A\ref{diameter}. To simply the notations, we will also assume that A\ref{diameter} holds true with $C=1$. We will demonstrate now how to approximate a geometric ball $B_r(x)$ by a union
of $n$-cylinders.
For every $x \in M$, $r>0$ and $k\in \mathbb{N}$ define
$$
U^-(x,r,k) = \bigcup_{A\in\cA^k,\, A\subset B_r(x)} A\subset B_r(x)
$$
 which is the union of all $k$-cylinders contained in $B_r(x)$. Since $\cA$ is generating, $U^-(x,r,k)\ne\emptyset$ for sufficiently large $k$.

Also define
$$
U^+(x,r,k) = \bigcup_{A\in\cA^k, \,A\cap B_r(x)\ne\emptyset} A,
$$
Then, $B_r(x)\subset U^+(x,r,k) \subset B_{r+k^{-p}}(x)$ due to Assumption A1.

The difference between $U^-(x,r,k)$ and $U^+(x,r,k)$ is
$$
U^+(x,r,k)\setminus U^-(x,r,k)=\bigcup_{A\in\cA^k, \,A\cap B_r(x)\ne\emptyset,\, A\not\subset B_r(x)} A.
$$
In order to get a nice approximation of $B_r(x)$ by $U^\pm(x,r,k)$, we need to make the diameter of $k$-cylinders small comparing to the size of the ball. For this purpose, we fix some $w>1$ that will be determined in Section~\ref{sec7} and put
$$
n=n(r) = \floor{r^{-\frac{w}{p}}}+1.
$$
This guarantees that $\diam\cA^{n} < n(r)^{-p}<r^w$. It follows that
$$
U^+(x,r,n(r))\setminus U^-(x,r,n(r)) \subset B_{r+r^w}(x)\setminus B_{r-r^w}(x).
$$
In particular, this shows that
$$
B_{r-r^w}(x)\subset U^-(x,r,n(r))\subset B_r(x),
$$
and the difference between $U^-(x,r,n(r))$ and $B_r(x)$ is  small:
\begin{equation}\label{e.approximation}
\mu(B_r(x)\setminus U^-(x,r,n(r))) < \mu(B_{r}(x)\setminus B_{r-r^w}(x))\le C\, r^{wa-b}\, \mu(B_r(x)),
\end{equation}
according to Assumption A\ref{regularity}. Similarly we have
$$
B_{r}(x)\subset U^+(x,r,n(r))\subset B_{r+r^w}(x),
$$
with difference given by
\begin{equation}
\mu(U^+(x,r,n(r))\setminus B_r(x)) < \mu(B_{r+r^w}(x)\setminus B_{r}(x))\le C\, r^{wa-b}\, \mu(B_r(x)).
\end{equation}

\section{Recurrence rate for metric balls and estimate on the short return.}\label{sec6}\setcounter{equation}{0}
Define the lower recurrence rate for points $x$ and $y$ by
$$
\underline{R}(x,y) =\liminf_{r\to 0}\frac{\log \tau_{B_r(y)}(x)}{-\log r}
$$
and set $\underline{R}(x) = \underline{R}(x,x)$. The latter was studied in ~\cite{RS, PS,GRS}.
The following Theorem was first stated in~\cite{RS} for fast mixing systems (systems that has super-polynomial decay of correlation), and  for Young's towers with polynomial tails in \cite{PS} as Theorem 6.2.
\begin{theorem}\label{6.1}
Assume that Assumptions A1--A4 are satisfied. Then $\underline{R}(x) = \dim_H\mu$, $\mu$-almost everywhere.
\end{theorem}

In order to prove this theorem, we will need the following de-correlation lemma.

\begin{lemma}\label{6.2}
There exists $C'>0$ such that for every $x\in M$, $r>0$ and integers $0\le k \le N$,
$$
\mu\{ y\in B_r(x):\, \tau_{B_r(x)}(T^k(y))<N\} \le C'\mu(B_{r+s}(x))k^{-\beta} + 2(N-k)\mu(B_{r+s}(x))^2,
$$
where $s=C(k/2)^{-p}$.
\end{lemma}
\begin{proof}
Recall that $B_r(x)\subset U^+(x,r,k/2)\subset B_{r+s}(x)$ with $s=C(k/2)^{-p}$. Also note that
\begin{align*}
\{y:\tau_{B_r(x)}(T^k(y))<N\} =& T^{-k}\{y:\tau_{B_r(x)}(y)<N-k\}\\\subset& T^{-k}\{y:\tau_{U^+(x,r,k/2)}(y)<N-k\}.
\end{align*}
By the $\phi$-mixing property, we get
\begin{align*}
&\mu\{ y\in B_r(x):\, \tau_{B_r(x)}(T^k(y))<N\}\\
\le& \mu(U^+(x,r,k/2)\cap T^{-k}\{y:\tau_{U^+(x,r,k/2)}(y)<N-k\} )\\
\le& \mu(U^+(x,r,k/2))\left(\mu\{ y:\,\tau_{U^+(x,r,k/2)}(y)<N-k\}+\phi(k/2)\right)\\
\le& \mu(B_{r+s}(x))\left(\mu\{ y:\,\tau_{B_{r+s}(x)}(y)<N-k\}+\phi(k/2)\right)\\
\le& C\mu(B_{r+s}(x))(k/2)^{-\beta} + (N-k)\mu(B_{r+s}(x))^2,
\end{align*}
where in the last inequality we use the coarse estimate $\mu(\tau_A<n) \le n\mu(A)$ for every measurable set $A$.

Now the proof of Theorem~\ref{6.1} follows the lines of Theorem 6.2 in~\cite{PS} and Lemma~\ref{6.2} above, which replaces Proposition 5.1 in \cite{PS}. Note that although in~\cite{PS} $T$ is assumed to be an invertible map, the proof of Theorem 6.2 in~\cite{PS} does not require that. In fact, the argument in~\cite{PS} is taken from Lemma 16 in~\cite{RS} which is stated for any measure preserving system.
\end{proof}

Next, define the hitting time exponent for the pair $(x,y)\in M\times M$ to be
$$
R_{x,y} = \min\{\underline{R}(x), \underline{R}(y),\underline{R}(x,y),\underline{R}(y,x)\}.
$$
Combining Theorem~\ref{6.1} with Proposition 1 in~\cite{GRS}, we immediately obtain the following lower bound on the hitting time exponent :
\begin{proposition}\label{6.3}
Assume that Assumptions A1--A4 hold true. Then $R_{x,y} \ge \dim_H\mu$ for $\mu\times\mu$-a.e.
$(x,y)\in M\times M$.
\end{proposition}

When trying to apply the arguments from Section~\ref{sec4} to geometric balls, the major difficulty arises in
obtaining an appropriate estimate for short returns, namely for $b_2$ in~(\ref{b2}). On one hand, we need large
 $n$ for the approximation of metric balls by $n$-cylinder sets, as observed in Section~\ref{sec5}, and on the other hand, if we put
$$
\pi(B_r(x),B_r(y)) = \min\{k \ge 0: T^{-k}B_r(x)\cap B_r(y)\ne\emptyset \mbox{ or } T^{-k}B_r(y)\cap B_r(x)\}
$$
and define $\ka_{B_r(x),B_r(y)}$ as in the case of cylinders,
then it is easy to see that $\ka_{B_r(x),B_r(y)}$ is of order $|\log r|$, which is much smaller than $n(r)= \floor{r^{-\frac{w}{p}}}+1$. In particular, we cannot expect to have sufficient de-correlation before the first return happens. To solve this issue, we make the following crucial observation concerning the $b_2$ term in
the estimates from \cite{AGG}.
\begin{remark}\label{6.4}
The $b_2$ term in (\ref{b2}) was used in Theorem 3 of \cite{AGG} to control terms of the form
$$
\sum_{\alpha\in I}E(X_\alpha[f_i({\bf W_\alpha+\bf e_i}) - {f_i}({\bf V_\alpha+e_i})]),
$$
where ${\bf W_\alpha}$ and ${\bf V_\alpha}$ are certain random $d$-vectors which are different only if
$\sum_{\alpha\ne\be\in B_\alpha}X_\beta\geq 1$, $X_\al$ with $\al\in I$ are Bernoulli random variables,
$\al\in B_\al
\subset I,\,\forall\al\in I$ and $\bf e_i$ is the unit vector with 1 on the $i$th place. Then we can write
\begin{eqnarray*}
&\sum_{\alpha\in I}E(X_\alpha[f_i({\bf W_\alpha+\bf e_i}) - {f_i}({\bf V_\alpha+e_i})])
\le 2\|f_i\| \sum_{\alpha\in I}E(X_\alpha\bbI_{\bf W_\alpha\ne\bf V_\alpha})\\
&\le 2\|f_i\| \sum_{\alpha\in I}P(X_\alpha=1,\,\sum_{\alpha\ne\be\in B_\alpha}X_\beta\geq 1).
\end{eqnarray*}
The term $b_2$ in \cite{AGG} is obtained by estimating
\[
P(X_\al=1,\,\sum_{\alpha\ne\be\in B_\alpha}X_\beta\geq 1)\leq\sum_{\alpha\ne\be\in B_\alpha}P(X_\al=1,\, X_\be=1)
\]
which is too coarse for our purposes. Thus it is possible to replace $b_2$ by
\begin{equation}\label{b'2}
b'_2 = \sum_{\alpha\in I}P(X_\alpha=1,\,\sum_{\alpha\ne\be\in B_\alpha} X_\beta\geq 1)
\end{equation}
 and (\ref{a5}) will still hold with $b'_2$ in place of $b_2$. It turns out, that we will be able to estimate
 $b'_2$ in a better way sufficient for our proof.
\end{remark}

Let
$$
X_{k,0}=\bbI_{B_r(y)}\circ T^k\text{ and }X_{k,1}=\bbI_{B_r(x)}\circ T^k,
$$
and for fixed $k\in\mathbb{N}, \al = 0,1$ and $M,R>0$ let
$$
B^{M,R}_{k,\al}=\{(l,\al),(l,1-\al):\, 0\leq l\leq M,\, |l-k|\leq R\}
$$
be the same as in Section~\ref{sec4}, with $U=B_r(x)$ and $V=B_r(y)$.
 The following proposition gives an (optimal) estimate on $b'_2$, which will be used in the next section.

\begin{proposition}\label{6.5}
Assume that Assumptions A1--A4 hold true. Then for every $\al=0,1$, positive integers $M$ and $k$, positive
real number $\sigma<\dim_H\mu$ and $\mu\times\mu$-a.e. $(x,y)$ such that the ratio $\frac {\mu(B_r(x))}{\mu(B_r(y))}$
remains sandwiched between two positive constants as $r\to 0$,
$$
\mu\{z:\, X_{k,\alpha}(z)=1,\,\sum_{(k,\alpha)\ne(l,\beta)\in B^{M,r^{-\sigma}}_{k,\al}}X_{l,\beta}\geq 1\}\leq \up_{x,y}(\mu(B_r(x)))
$$
where $\up(u)=\up_{x,y}(u)=o(u)$ as $u\to 0$. In particular, if $M=M(r)\to\infty$ is such that $M(r)\up_{x,y}(\mu(B_r(x)))\to
0$ as $r\to 0$ then
\begin{eqnarray*}
&b'_2=b'_2(x,y)=\sum_{(k,\al)\in I_M}\mu\{ z:\, X_{k,\alpha}(z)=1,\,\sum_{(k,\alpha)\ne(l,\beta)
\in B^{M,r^{-\sigma}}_{k,\al}}X_{l,\beta}\geq 1\}\\
&\le M(r)\up_{x,y}(\mu(B_r(x)))\to 0, \,\mbox{ as } r \to 0.
\end{eqnarray*}
\end{proposition}
\begin{proof}
We will only consider the case $\al=1$. The case $\al=0$ is similar by switching $x$ and $y$. Note that $\rho(x,y,r)\to \rho(x,y)\in(0,1)$ as $r\to 0$ allows us to interchange between $\up(\mu(B_r(x)))$ and $\up(\mu(B_r(y)))$.

\noindent {\em Case 1.} $\al=\be=1$.

First note that when $\al=\be=1$,
\begin{eqnarray*}
&\mu\{ z:\, X_{k,1}(z)=1,\,\sum_{(k,1)\ne(l,1)\in B^{M,r^{-\sigma}}_{k,1}}X_{l,1}\geq 1\}\\
&=\mu\big( (T^{-k}B_r(x))\cap(\bigcup_{k+r^{-\sig}\geq l\geq 0\vee(k-r^{-\sig})}T^{-l}B_r(x)\big).
\end{eqnarray*}
We split this into two parts that correspond to $l<k$ and $l>k$, respectively,
\begin{equation*}
b^-=b^-_k(x)=\mu\left(T^{-k}B_r(x)\cap\bigcup_{l=1}^{k\wedge r^{-\sigma}}T^{-(k-l)}B_r(x)\right),
\end{equation*}
and
$$
b^+=b^+_k(x)=\mu\left(T^{-k}B_r(x)\cap\bigcup_{l=1}^{r^{-\sigma}}T^{-(k+l)}B_r(x)\right).
$$
\noindent(i) First, we will prove that $b^-\leq b^+$ and then it will remain only to estimate $b^+$.
It suffices to show that
$$
\mu\left(T^{-k}B_r(x)\cap\bigcup_{l=1}^{k\wedge r^{-\sigma}}T^{-(k-l)}B_r(x)\right)=\mu\left(T^{-k}B_r(x)\cap\bigcup_{l=1}^{k\wedge r^{-\sigma}}T^{-(k+l)}B_r(x)\right).
$$
For this purpose, we write $J_l = T^{-k}B_r\cap T^{-(k-l)}B_r$ and $U = \bigcup_{l=1}^{k\wedge r^{-\sigma}}J_l$. Similarly
$\tilde{J}_l = T^{-n}B_r(x)\cap T^{-(k+l)}B_r(x)$ and $\tilde{U} = \bigcup_{l=1}^{k\wedge
r^{-\sigma}}\tilde{J}_l$. In order to show that $\mu(U) = \mu(\tilde{U})$, we decompose $U$ into a
disjoint union
$$
U = \bigcup_{l=1}^{k\wedge r^{-\sigma}} J'_l\quad\mbox{where}\quad J'_l = J_l\setminus \bigcup_{j=1}^{l-1}
J_l\cap J_j.
$$
Similarly, $\tilde{J}'_l=\tilde{J}_l\setminus \bigcup_{j=1}^{l-1}\tilde{J}'_l\cap \tilde{J}_j$ decomposes $\tilde{U}$ into a disjoint union.

Note that $T^{-l}J_l=\tilde{J}_l$ when $l\leq k$, and so
$$
T^{-l}J'_l = T^{-l}J_l\setminus \bigcup_{j=1}^{l-1}T^{-l}(J_l\cap J_k)=\tilde{J}_l\setminus \bigcup_{j=1}^{l-1}(\tilde{J}_l\cap \tilde{J}_k)=\tilde{J}'_l.
$$
This shows that $\mu(J'_l)=\mu(\tilde{J}'_l)$ when $l\leq k$. Summing over $l$, we get $\mu(U)=\mu(\tilde{U})$.

\noindent(ii) Next, we estimate $b^+$ using the Lebesgue density point technique employed in
Proposition 7.1 from~\cite{PS}.
Fix $\varepsilon_1,\varepsilon_2>0$. Consider the "good" set:
$$
G_{\sigma, r_1} = \{z\in\bfM: \forall r < r_1, \tau_{B_{2r}(z)}(z)\ge r^{-\sigma}\}.
$$
By Theorem~\ref{6.1}, $G_{\sig,r_1}\uparrow G_\sig$ as $r_1\downarrow 0$ where $\mu(\bfM\setminus G_\sig)=0$
and, in particular, we can take $r_1$ small enough such that $\mu(\bfM\setminus G_{\sigma, r_1})<\varepsilon_1$.

Let $\tilde{G}_{\sigma,r_1}$ be the set of Lebesgue density points of $G_{\sigma,r_1}$ with respect to the measure $\mu$, that is, the set of points $z\in G_{\sigma,r_1}$ such that
$$
\lim_{r\to 0}\frac{\mu(B_r(z)\cap G_{\sigma,r_1})}{\mu(B_r(z))}=1.
$$
$\tilde{G}_{\sigma,r_1}$ has full measure in $G_{\sigma,r_1}$. Then we take $r_2 < r_1$, such that
$$
\tilde{G}_{\sigma,r_1,r_2} := \left\{z:\forall r < r_2,\frac{\mu(B_r(z)\cap G_{\sigma,r_1})}{\mu(B_r(z))}>1-\varepsilon_2\right\}
$$
satisfies
$$
\mu(G_{\sigma,r_1}\setminus\tilde{G}_{\sigma,r_1,r_2})<\varepsilon_1.
$$
Now, for $x\in\tilde{G}_{\sigma,r_1,r_2}$ and $r<r_2$,
\begin{eqnarray*}
&b^+=b^+_k(x) =\mu(z\in T^{-k}B_r(x): \tau_{B_r(x)}(T^kz)<r^{-\sigma})\\
&=\mu(z\in B_r(x): \tau_{B_r(x)}(z)<r^{-\sigma})\\
&\le\mu(z\in B_r(x): \tau_{B_{2r}(z)}(z)<r^{-\sigma})\le\mu(B_r(x)\setminus G_{\sigma,r_1})
\le\varepsilon_2\mu(B_r(x)),
\end{eqnarray*}
which fails only on the set $\bfM\setminus\tilde G_{\sig,r_1,r_2}$ with measure less than
$2\varepsilon_1$ and which decreases as $r_1,r_2\downarrow 0$ to a set of $\mu$-measure 0.

\noindent {\em Case 2.} $\al=1, \be=0$.

First note that when $l=k$ then $T^{-k}B_r(x)\cap T^{-l}B_r(y)\ne\emptyset$ if and only if $B_r(x)\cap B_r(y)\ne\emptyset$, which can be avoided if $r$ is taken small enough. As in the previous case, we have  two sub-cases: $l<k$ and $l>k$. Denote by
$$
c^-(x,y)=\mu\left(T^{-k}B_r(x)\cap\bigcup_{l=1}^{k\wedge r^{-\sig}}T^{-(k-l)}B_r(y)\right),
$$
and
$$
c^+(x,y)=\mu\left(T^{-k}B_r(x)\cap\bigcup_{l=1}^{r^{-\sig}}T^{-(k+l)}B_r(y)\right).
$$
 By an argument similar to the $b^-$ case (with one of the $x$ replaced by $y$) we see
  that $c^-(x,y)\leq c^+(y,x)$, where the latter can be seen as part of the case $\al=0,\be =1$ and $l>k$. Hence,
  as before, we only need to estimate  $c^+(x,y)$ and the proposition follows.

By Proposition~\ref{6.3} and the Fubini theorem, for almost every $y\in\bfM$ we have $\underline{R}(x,y)\ge \dim_H\mu$ for almost every $x$. For such fixed $y$, Consider the "good" set
$$
G^y_{\sigma, r_1} = \{z\in\bfM: \forall r < r_1,\, \tau_{B_{r}(y)}(z)\ge r^{-\sigma}\}.
$$
Again, $G^y_{\sig,r_1}\uparrow G^y_\sig$ as $r_1\downarrow 0$, and so
we can take $r_1$ small enough so that $\mu(\bfM\setminus G^y_{\sigma, r_1})<\varepsilon_1$.

The rest of the proof proceeds in the same way as in the case $b^+$. We take  $\tilde{G}^y_{\sigma,r_1}$ to be
the set of Lebesgue density point of $G^y_{\sigma,r_1}$,
$$
\tilde{G}^y_{\sigma,r_1}=\left\{ z:\, \lim_{r\to 0}\frac{\mu(B_r(z)\cap G^y_{\sigma,r_1})}{\mu(B_r(z))}
=1\right\}.
$$
Then we take $r_2 < r_1$ so that the set
$$
\tilde{G}^y_{\sigma,r_1,r_2} := \left\{z:\forall r < r_2,\frac{\mu(B_r(z)\cap G^y_{\sigma,r_1})}{\mu(B_r(z))}>1-\varepsilon_2\right\}
$$
satisfies
$$
\mu(G^y_{\sigma,r_1}\setminus\tilde{G}^y_{\sigma,r_1,r_2})<\varepsilon_1.
$$
Now, for $x\in\tilde{G}^y_{\sigma,r_1,r_2}$ and $r<r_2$,
\begin{eqnarray*}
&c^+=c^+(x,y) =\mu(z\in T^{-n}B_r(x): \tau_{B_r(y)}(T^nz)<r^{-\sigma})\\
&= \mu(z\in B_r(x): \tau_{B_r(y)}(z)<r^{-\sigma})\le\mu(B_r(x)\setminus G^y_{\sigma,r_1})
\le\varepsilon_2\mu(B_r(x))
\end{eqnarray*}
concluding the argument as in the previous case and completing the proof of Proposition~\ref{6.5}.
\end{proof}

Recall that the sets $U^\pm(x,r,n(r))$ are defined in Section~\ref{sec5} as the approximation of $B_r(x)$ by $n(r)$-cylinders from inside and outside, respectively. Since $U^-(x,r,n(r))\subset B_r(x)$ and $U^-(y,r,n(r))\subset B_r(y)$, it follows that
$$
\tilde{X}_{k,0}=\bbI_{U^-(y,r,n(r))}\circ T^k\leq X_{k,0}\text{ and }\tilde{X}_{k,1}=\bbI_{U^-(x,r,n(r))}
\circ T^k\leq X_{k,1}.
$$
Hence,
$$
\{\sum_{(l,\beta)\in B^{M,r^{-\sigma}}_{n,\al}}\tilde{X}_{l,\beta}\geq 1\} \subset \{\sum_{(l,\beta)\in B^{M,r^{-\sigma}}_{n,\al}}X_{l,\beta}\geq 1\},
$$
and so
\begin{equation}\label{tb'2}
\tilde b'_2=\tilde b'_2(x,y)=\sum_{(k,\al)\in I_M}\mu\{ z:\,\tilde X_{k,\alpha}(z)=1,\sum_ {(k,\alpha)\ne(l,\be)\in B^{M,r^{-\sigma}}_{n,\al}}\tilde{X}_{l,\beta}(z)\geq 1\}\leq b'_2.
\end{equation}
This together with Proposition~\ref{6.5} yields
\begin{proposition}\label{6.6}
Assume that Assumptions A1--A4 hold true, and $\rho(x,y,r)\to \rho(x,y)\in(0,1)$ as $r\to 0$. Then for every
$\al=0,1$, a positive integers $M$, positive real number $\sigma<\dim_H\mu$ and $\mu\times\mu$-a.e. $(x,y)$,
\[
\tilde{b}'_2\leq M\up(\mu(B_r(x))
\]
where $\up(u)=o(u)$ as $u\to 0$.
\end{proposition}

\section{Proof of Theorems \ref{t.phimixing}}\label{sec7}\setcounter{equation}{0}

Recall that $n(r) = \floor{r^{-\frac{w}{p}}}+1$ where $w>1$ will be determined later.
We will write $\tau = \tau_{B_r(y)}$, $\tau_M = \min(\tau,M)$, $\tilde{\tau} = \tau_{U^-(y,r,n(r))}$ and  $\tilde{\tau}_M = \min(\tilde{\tau},M)$. Then, $\tilde{\tau}\ge \tau$ and $\tilde{\tau}_M\ge \tau_M$.
Introduce the sums
$$
S_M(z) = \sum_{j=0}^{M}X_{k,1}(z)\quad\mbox{and}\quad\tilde{S}_M(z) =\sum_{j=0}^M\tilde X_{k,1}(z)
$$
which count the number of visits to $B_r(x)$ and to $U^-(x,r,n(r))$, respectively. Set also
$$
\tilde{\rho} = \frac{\mu(U^-(y,r,n(r)))}{\mu(U^-(x,r,n(r)))+\mu(U^-(y,r,n(r)))}.
$$
In other words, every term with tilde is defined using the approximations $U^-(x,r,n(r))$ and $U^-(y,r,n(r))$.
Similarly to Section \ref{sec4} we introduce also a sequence of independent Bernoulli random variables
$\{ Y_{k,\al}:\, k\geq 0,\,\al=0,1\}$ such that $Y_{k,\al}$ has the same distribution as $\tilde X_{k,\al}$
and we set
\[
\tilde S^*_M=\sum_{k=0}^{M-1}Y_{k,1},\,\tilde\tau^*=\min\{ k\geq 0:\, Y_{k,0}=1\}\,\,\,\mbox{and}\,\,\,
\tilde\tau^*_M=\min(\tau,M).
\]
As in (\ref{a0}) we estimate
\begin{equation*}
d_{TV}(\cL(\Sxy),Geo(\rho(x,y,r)))\le B_1+B_2+B_3+B_4+B_5,
\end{equation*}
where
\begin{eqnarray*}
&B_1= d_{TV}(\cL(S_\tau),\cL(S_{\tilde{\tau}_M})),\, B_2= d_{TV}(\cL(S_{\tilde{\tau}_M}),\cL(\tilde{S}_{\tilde{\tau}_M})),\\
&B_3= d_{TV}(\cL(\tilde{S}_{\tilde{\tau}_M}),\cL(\tilde{S}^*_{\tilde{\tau}^*_M})),\,
B_4= d_{TV}(\cL(\tilde{S}^*_{\tilde{\tau}^*_M}),\cL(\tilde{S}^*_{\tilde{\tau}^*})),\\
&B_5= d_{TV}(Geo(\tilde{\rho}),Geo(\rho)).\\
\end{eqnarray*}
 Note that in $B_1$ we use $\tilde{\tau}_M$ and not $\tau_M$ as in (\ref{a0}) but this will not make
  essential difference.

To obtain an estimate on $B_1$, we write
$$
B_1\le d_{TV}(\cL(S_\tau),\cL(S_{{\tau}_M}))+d_{TV}(\cL(S_{\tau_M}),\cL(S_{\tilde{\tau}_M})).
$$
The first term is similar to $A_1$ from (\ref{a0}) and is bounded by $\mu\{ z:\,\tau(z)>M\}$. For the second term, note that $\tau_M\le\tilde{\tau}_M\le M$. If $\tau_M\ne \tilde{\tau}_M$, then the first entry to $B_r(y)$ must be contained in $B_r(y)\setminus U^-(y,r,n(r))$. Therefore
\begin{eqnarray*}
&d_{TV}(\cL(S_{\tau_M}),\cL(S_{\tilde{\tau}_M})) \le\mu\{ z:\,\tau_{B_r(y)\setminus U^-(y,r,n(r))}(z)<M\}\\
&\le \mu\left(\bigcup_{j=0}^{M-1}T^{-j}(B_r(y)\setminus U^-(y,r,n(r)))\right)\\
&\le M\mu(B_r(y)\setminus U^-(y,r,n(r)))\le Mr^{wa-b}\mu(B_r(y)),
\end{eqnarray*}
where the last inequality follows from (\ref{e.approximation}).

Next, $B_5$ is the same as $A_4$ in Section~\ref{sec4}, and so
$$
B_5 \le \frac{2|\tilde{\rho}-\rho|}{\tilde{\rho}\rho}.
$$

To estimate $B_2$, note that $U^-(x,r,n(r))\subset B_r(x)$ which implies that $\tilde{S}_M\le S_M$. Therefore
\begin{eqnarray*}
&B_2\le \sum_{k=0}^M|\mu\{ z:\, S_{\tilde{\tau}_M}(z) = k\} - \mu\{ z:\,\tilde{S}_{\tilde{\tau}_M}(z) = k\}|\\
&\le \sum_{k=0}^M\mu\{ z:\, S_{\tilde{\tau}_M}(z) > k, \tilde{S}_{\tilde{\tau}_M}(z) = k\}.
\end{eqnarray*}
Since $U^-(x,r,n(r))\subset B_r(x)$, all the extra visits $T^jz\in B_r(x)$, which make $S_{\tilde{\tau}_M}$
larger than $\tilde S_{\tilde{\tau}_M}$, must be contained in $ B_r(x)\setminus U^-(x,r,n(r))$. Thus,
by Assumption A\ref{regularity},
\begin{equation}
B_2\le\sum_{k=0}^M(M-k)\mu(B_r(x)\setminus U^-(x,r,n(r)))\le \frac12 M^2r^{wa-b}\mu(B_r(x)).
\end{equation}

Observe that $B_3$ and $B_4$ are the same as $A_2$ and $A_3$ in Section~\ref{sec4} but we
estimate $B_3$ by (\ref{a4}) and (\ref{a5})
with $b_2$ replaced by $\tilde{b}'_2$ in view of Remark~\ref{6.4}. Hence,
$$
B_4 \le \mu\{ z:\,\tilde{\tau}^*(z) >M\} = (1-\mu(U^-(y,r,n(r))))^{M+1},
$$
and setting $p_x = \mu(U^-(x,r,n(r)))$ and $p_y = \mu(U^-(y,r,n(r)))$ we obtain using
Proposition \ref{6.6} that
\[
B_3\le 8MR(p_x+p_y)^2 + 4M\up(\mu(B_r(x)))+12M\phi(R-n) +4M(p_x^2+p_y^2).
\]
Combining the above estimates we obtain
\begin{eqnarray}\label{ineq1}
&d_{TV}(\cL(\Sxy),Geo(\rho))\le\mu\{ z:\,\tau(z)>M\}+ Mr^{wa-b}\mu(B_r(y))\\
&+\frac 12 M^2r^{wa-b}\mu(B_r(x))+4MR(p_x+p_y)^2+ 2M\up(\mu(B_r(x)))+6M\phi(R-n(r))\nonumber\\ &+2M(p_x^2+p_y^2)+(1-\mu(U^-(y,r,n(r))))^{M+1}+\frac{2|\tilde{\rho}-\rho|}{\tilde{\rho}\rho}.\nonumber
\end{eqnarray}

Next, observe that $\tilde\tau\geq\tau$, and so $\{ z:\,\tilde\tau(z)>M\}\supset\{ z:\,\tau(z)>M\}$.
On the other hand,
\begin{eqnarray*}
&\{ z:\,\tilde\tau(z)>M\}\setminus\{ z:\,\tau(z)>M\}=\{ z:\,\tilde\tau(z)>M,\,\tau(z)\leq M\}\\
&\subset\cup_{k=0}^M\{ z:\, T^kz\in B_r(y)\setminus U^-(y,r,n(r))\}.
\end{eqnarray*}
Hence, by Assumption A\ref{regularity},
\begin{eqnarray}\label{ineq2}
&\mu\{ z:\,\tau(z)>M\}\leq\mu\{ z:\,\tilde\tau(z)>M\}+\sum_{k=0}^M\mu\big(B_r(y)\setminus U^-(y,r,n(r))\big)\\
&\leq\mu\{ z:\,\tilde\tau(z)>M\}+(M+1)r^{wa+b}\mu(B_r(y)).\nonumber
\end{eqnarray}
Similarly to (\ref{a1}) we obtain that
\begin{equation}\label{ineq3}
\mu\{ z:\,\tilde\tau(z)>M\}\leq (1-p_y)^{M+1}+d_{TV}(\cL(\tilde\bfX_M),\,\cL(\bfY_M))
\end{equation}
where $\tilde\bfX_M=\{\tilde X_{k,\al},\, 0\leq k\leq M,\,\al=0,1\}$ and $\bfY_M=\{ Y_{k,\al},\,
0\leq k\leq M,\,\al=0,1\}$. In the same way as in (\ref{a5}) with $\tilde b_2$ in place of $b_2$ we estimate
\begin{eqnarray}\label{ineq4}
&d_{TV}(\cL(\tilde\bfX_M),\,\cL(\bfY_M))\\
&\leq 4MR(p_x+p_y)^2+2M\up(\mu(B_r(x)))+6M\phi(R-n(r))+2M(p^2_x+p_y^2).\nonumber
\end{eqnarray}

When $wa-b>0$ and $\rho(x,y,r)\to \rho\in(0,1)$ as $r\to 0$ then by~(\ref{e.approximation}) the pairwise
ratios of $\mu(U^\pm(x,r,n(r)))$, $\mu(U^\pm(y,r,n(r)))$, $\mu(B_r(x))$ and $\mu(B_r(y))$ are sandwiched
between positive constants. This together with Proposition \ref{6.3} yields that for $\mu\times\mu$ almost
all $x,y$ we can choose $M=M_{x,y}(r)\to\infty$ as $r\to 0$ so that
\begin{eqnarray}\label{to0}
 &M^2_{x,y}(r)r^{wa-b}\mu(B_r(x))\to 0,\,\, M_{x,y}(r)\up(B_r(x)))\to 0,\\
&\big(1-\mu(U^-(y,r,n(r)))\big)^{M_{x,y}(r)}\to 0\,\,\mbox{and}\,\,\frac {|\tilde\rho-\rho|}{\tilde\rho\rho}
\to 0\,\,\mbox{as}\,\, r\to 0.\nonumber
\end{eqnarray}
Observe, that the assumption (\ref{c.parameters}) enables us to choose $w$ and $\sig$ so that $wa-b>d$, $d>\sig$, $\sig>\frac wp$ and $\sig\be>d$. Indeed, these inequalities will hold true if we choose $\sig=d-\ve$
and $w=dp-\del$ for $\ve>0$ and $\del>0$ satisfying $p\ve<\del<dp-\frac {b+d}a$ and $\ve<d(1-\be^{-1})$,
which is possible in view of (\ref{c.parameters}) and since $\be>1$. With such $w$ and $\sig$ take
$R=R(r)=r^{-\sig}$. Then, as $r\to 0$, for $\mu\times\mu$-almost all $(x,y)$ we can choose
$M_{x,y}(r)\to\infty$ so that both (\ref{to0}) and
\[
M_{x,y}(r)R(r)(p_x+p_y)^2\to 0\,\,\,\mbox{and}\,\,\, M_{x,y}(r)\phi(R(r)-n(r))\to 0,
\]
hold true. This completes the proof of Theorem \ref{t.phimixing} in view of (\ref{ineq1})--(\ref{to0}).    \qed

\section{Geometric law limit for suspensions}\label{sec8}
In this section we will prove Theorem~\ref{t.suspension}.
  Given a well-placed set $U\subset \Om$, the set $\{(x,0): (x,k)\in U\mbox{ for some } k\}\subset \Omega_0$ can be naturally identified with $\Pi(U)$ by dropping the second coordinate and it will also be denoted by $\tU$. 
\begin{lemma}\label{l.tsig} Let $(\Om,\mu,T)$ be a discrete time suspension over a measure preserving dynamical system 
$(\tilde\Om,\tilde\mu,\tilde T)$ and $U,\, V$ be a pair of well-placed sets such that their projections to the base $\tilde U=\Pi(U)$
and $\tilde V=\Pi(V)$ are disjoint. Introduce the set $W=W_U=\{(x,i)\in\Om:\, i>0$ and there exists $j\geq i$ such that $(x,j)\in U\}$.
Then $\Sig_{U,V}(x,i)=\Sig_{\tilde U,\tilde V}(x,i)$ for any $(x,i)\in\Om\setminus W$ and $\Sig_{U,V}(x,i)=\Sig_{\tilde U,\tilde V}(x,i)+1$ 
if $(x,i)\in W$.
\end{lemma}
\begin{proof}
By the definition of a discrete time suspension and taking into account that $\tilde U\cap\tilde V=\emptyset$, each time the orbit $T^k(x,i),\, k\geq 0$
visits $\tilde U$ it must visit $U$ before it can visit $\tilde V$ or $V$. After each visit of $U$ the orbit either visits $\tilde V$ and then $V$ which
stops the counting or it visits $\tilde U$ and then $U$ and the process repeats itself. Since $U$ is well-placed, any orbit must visit ones $\tilde U$
 between two successive visits of $U$. Thus, if an orbit visits $\tilde U$ for the first time before it ever visited $U$ then there will be the same
  number of visits of each of them before the first visit of $V$. On the other hand, if $(x,i)\in W$ then the orbit $T^k(x,i),\, k\geq 0$ visits first
  $U$ before it visits $\tilde U$, and so in this case there will be one more visit to $U$ than to $\tilde U$.
\end{proof}

Now we are ready to prove Theorem~\ref{t.suspension}.
\begin{proof}[Proof of Theorem~\ref{t.suspension}]
Let $U_n$ and $V_n$, $n\geq 1$ be decreasing sequences of measurable subsets of $\Om$ such that $\tilde U_n\cap\tilde V_n=\emptyset$ 
and $\tilde\mu(\tilde U_n),\tilde\mu(V_n)\to 0$
as $n\to\infty$. Then the set $W_{U_n}$ appearing in Lemma~\ref{l.tsig} satisfies $\mu(W_{U_n})\leq\int_{\tilde U_n}R(x)d\tilde\mu(x)\to 0$.
This together with Lemma~\ref{l.tsig} yields that Theorem ~\ref{t.suspension} will follow if we show that 
\begin{equation}\label{8.1a}
\Sig_{\tU_n,\tV_n}\xRightarrow{\mu}Geo(\rho)\hspace{1cm}\mbox{ if }\hspace{1cm}\tilde\Sig_{\tU_n,\tV_n}\xRightarrow{\tmu} Geo(\rho)
\end{equation}
where the first convergence in distribution is on the probability space $(\Om,\mu)$ while the second one is on the probability space 
$(\tilde\Om,\tilde\mu)$. Next, define the probability measure $\hat\mu$ on $\tilde\Om$ by setting 
\[
\hat\mu(\Gam)=(\int_{\tilde\Om}R(x)d\tilde\mu(x))^{-1}\int_\Gam R(x)d\tilde\mu(x)
\]
for any measurable set $\Gam\subset\tilde\Om$. Observe that for each $x\in\tilde\Om$ the number of visits of $U_n$ by the 
orbits $T^k(x,i),\, k\geq 1$ until the first visit of $\tilde V_n$ is the same for all $i=0,1,...,R(x)-1$. It follows that it
suffices to show that the convergence $\Sig_{\tU_n,\tV_n}\xRightarrow{\mu}Geo(\rho)$ is equivalent to the convergence
$\Sig_{\tU_n,\tV_n}\xRightarrow{\hat\mu}Geo(\rho)$ where the latter is considered on the probability space $(\tilde\Om,\hat\mu)$.
But if $\tilde\Sig_{\tU_n,\tV_n}\xRightarrow{\tmu} Geo(\rho)$ then $\Sig_{\tU_n,\tV_n}\xRightarrow{\hat\mu}Geo(\rho)$ holds true
by Theorem 1 from \cite{Z07}, completing the proof of Theorem~\ref{t.suspension}.
\end{proof}

\section{Applications}\label{sec9}
\subsection{Gibbs-Markov systems}
As the first example, we consider a Gibbs-Markov map $T$ on a Lebesgue space $(X, \mu)$. Recall that a map $T$ is called {\em Markov} if there is a countable measurable partition $\cA$ on $X$ with $\mu(A)>0$ for all $A\in \cA$, such that for all $A\in \cA$, $T(A)$ is injective and can be written as a union of elements in $\cA$. Write $\cA^n=\bigvee_{j=0}^{n-1}T^{-j}\cA$ as before, it is also assumed that $\cA$ is (one-sided) generating.

Fix any $\gamma\in(0,1)$ and define the metric $d_\gamma$ on $X$ by $d_\gamma(x,y) = \gamma^{s(x,y)}$, where $s(x,y)$ is the largest positive integer $n$ such that $x,y$ lie in the same $n$-cylinder. Define the Jacobian $g=JT^{-1}=\frac{d\mu}{d\mu\circ T}$ and $g_k = g\cdot g\circ T \cdots g\circ T^{k-1}$.

The map $T$ is called {\em Gibbs-Markov} if it preserves the measure $\mu$, and also satisfies the following two assumptions:\\
(i) The big image property: there exists $C>0$ such that $\mu(T(A))>C$ for all $A\in \cA$.\\
(ii) Distortion: $\log g|_A$ is Lipschitz for all $A\in\cA$.


 In view of (i) and (ii), there exists a constant $D>1$ such that for all $x,y$ in the same $n$-cylinder, we have the following distortion bound:
$$
\left|\frac{g_n(x)}{g_n(y)}-1\right|\le D d_\gamma(T^nx,T^ny),
$$
and the Gibbs property:
$$
D^{-1}\le \frac{\mu(A_n(x))}{g_n(x)}\le D.
$$
It is well known (see, for example, Lemma 2.4(b) in~\cite{MN05}) that Gibbs-Markov systems are exponentially $\phi$-mixing. Therefore we have the following corollaries of Theorem~\ref{thm2.1} and~\ref{t.phimixing}:
\begin{theorem}\label{t.9.1}
Let $T$ be a Gibbs-Markov map on $(X,\mu)$ with finite entropy $-\sum_{A\in\cA}\mu(A)\log \mu(A)<\infty$.
Let $\{m(n)\}_{n\ge1}\subset\mathbb{N}\setminus\{0\}$ be a sequence satisfying $|m(n)-n|=o(n)$
as $n\rightarrow\infty$. Then for $\mu\times \mu$-a.e. $(\omega,\eta)\in X\times X$,
\begin{equation*}
\underset{n\to\infty}{\lim}\:d_{TV}(\mathcal{L}(\Sig^{\om,\eta}_{n,m(n)}),
Geo(\frac{\mu(A_{m(n)}^{\eta})}{\mu(A_{m(n)}^{\eta})+\mu(A_{n}^{\omega})}))=0\:.
\end{equation*}
In particular, if
\begin{equation*}
\lim_{n\to\infty}\frac {\mu(A_n^\om)}{\mu(A^\eta_{m(n)})}=\la
\end{equation*}
then $\cL(\Sig_{n,m(n)}^{\om,\eta})$ converges in total variation as
$n\to\infty$ to the geometric distribution with the parameter
$(1+\la)^{-1}$.
\end{theorem}

\begin{theorem}\label{t.9.2}
	Let $X$ be a compact metric space. Assume that there exists a countable generating partition $\cA$, such that the measure preserving system $(T,\mu)$ is Gibbs-Markov. Also assume that Assumptions A1, A3 and A4 are satisfied with
	\begin{equation*}
	p>\frac{d+b}{ad}.
	\end{equation*}
	Then for $\mu \times \mu$ almost every $(x,y)\in\bfM\times\bfM$, such that $\rho(x,y,r)\to \rho(x,y)\in(0,1)$ as
	$r\to 0$,
	$$
	\lim_{r\to 0}d_{TV}(\cL(\Sxy),Geo(\rho(x,y))) = 0.
	$$
\end{theorem}

\begin{remark}
	In many cases (for example, those in~\cite{Go}), the metric $d_\gamma$ is equivalent to the metric of $X$. Then Assumption A1 is satisfied with $\diam \cA^n \le \gamma^n$, and the assumption $p>\frac{d+b}{ad}$ is redundant.
\end{remark}

\subsection{Conformal repellers} A conformal repeller
is a maximal compact set $\Omega\subset M$ so that $T$ acts conformally
on $\Omega$ and is expanding, that is there exists a $\beta>1$ so that
$|DT^kv|\ge \beta^k$ for all large enough $k$ and all $v\in T_xM\,\forall x\in\Omega$.

\begin{theorem}
Let $\Omega\subset M$ be a conformal repeller for the $C^{1+\alpha}$-map $T:M\circlearrowleft$
and let $\mu$ be an equilibrium state for a H\"older continuous potential
$f:\Omega\to\mathbb{R}$. Then the conclusion of Theorem~\ref{t.phimixing} is valid.
\end{theorem}

\begin{proof}
We verify Assumptions A1--A4 using the fact that Markov partitions $\mathcal{A}$ of arbitrarily small diameter
 can be constructed here. Let $\mathcal{A}$ be such a partition. Then:\\
Assumption A1 holds since the map is uniformly expanding, so $\diam\cA^n\le \beta^{-n}$.\\
Assumption A2 follows from the fact that the equilibrium states are $\psi$-mixing with respect to the partition $\cA$ at exponential speed, thus left $\phi$-mixing.\\
Assumption A3 is shown in Theorem 21.3 of~\cite{P97} that every equilibrium state on a conformal repeller has exact dimension.\\
Assumption A4 is satisfied for any $w>1$ as $\mu$ is diametrically regular~\cite{PW97} and thus also has the annular decay property~\cite{Bu99}. This yields $\frac{\mu(B_{r+r^w}(x)\setminus B_r(x))}{\mu(B_r(x))}=\mathcal{O}(r^{(w-1)\delta})\to0$
 for some $\delta>0$. \\
\end{proof}

\subsection{Young's towers} First introduced by Young~\cite{Yo1},~\cite{Yo2}, Young's towers (or {\em Gibbs-Markov-Young structure}) is a useful tool to study the statistical property of $C^{1+\alpha}$ systems that are non-uniformly hyperbolic. When the systems is non-invertible, a Young's tower can be viewed as a discrete time suspension over a Gibbs-Markov system,\footnote{When $T$ is a diffeomorphism,  in order to obtain a Gibbs-Markov system one needs to take the quotient over the stable leaves.} such that the roof function $R$ (in this case, it is usually call the {\em return time function}) is integrable with respect to the measure $\tmu$.\footnote{Alternatively one can assume that the roof function is integrable w.r.t. a reference measure $m$, for which the map $\tT$ has ``good'' Jacobian. Under these conditions, it is shown in~\cite{Yo2} that there exists an invariant measure $\tmu$, absolutely continuous w.r.t. $m$.}
As an immediate corollary of Theorem~\ref{t.suspension},~\ref{t.9.1} and~\ref{t.9.2}, we have:

\begin{theorem}
	Let $(\Om, \mu, T)$ be a Young's tower such that the return time function $R$ is integrable. Assume that $\{U_n\}, \{V_n\} $ are two sequences of well-placed sets. 
	
(i) If $\tU_n = A^\om_n, \tV_n = A^\eta_{m(n)}$ for some $\om, \eta$ and $m(n)$ satisfy the assumptions of Theorem~\ref{t.9.1} and the base system 
has finite entropy, then $\Sig_{U_n,V_n}$ converges in distribution to the geometric distribution $Geo((1+\la)^{-1})$, where $\lambda$ is given by
	\begin{equation*}
	\lim_{n\to\infty}\frac {\mu(A_n^\om)}{\mu(A^\eta_{m(n)})}=\la;
	\end{equation*}
	
 (ii) If the base system $(\tom,\tmu,\tT)$ satisfies A1, A3 and A4, and $\tU_n = B_{r_n}(x), \tV_n = B_{r_n}(y)$ satisfy the assumptions of Theorem~\ref{t.9.2} for some sequence of positive real numbers $\{r_n\}$ with $r_n\to 0$, then $\Sig_{U_n,V_n}$ converges in distribution to  $Geo(\rho(x,y))$.
\end{theorem}

There is plenty of dynamical systems which can be modeled by Young's towers with the integrable return map being the first return map to the set which
 serves as a base. Then the dynamical system and its Young tower representation are isomorphic and limit theorems for both systems are equivalent.
 The corresponding one dimensional examples include uniform expanding piecewise $C^2$-map  with the Markov property, Gauss map, Pomeau--Manneville maps 
 (also known as the intermittent map) and certain unimodal maps (see, for instance, \cite{MN05}). For a more general construction in higher dimensions see \cite{Go}.
 
 \subsection{A probabilistic application}
 Let $X_n,\, n\geq 0$ be a Markov chain on a countable state space $\cA$ with an invariant probability measure $\mu$.
 Suppose that $X$ is ergodic, aperiodic and satisfies the Doeblin condition: $\exists\Gam\subset\cA$ with $\mu(\Gam)=1$,
 $\exists\ve\in(0,1),\,\exists n\geq 1$ such that $\forall c\in\Gam,\,\forall\Del\subset\cA$ with $\mu(\Del)<\ve$ the
 $n$-step transition probability satisfies $P(n,x,\Del)=P\{ X_n\in\Del\,|\, X_0=x\}\leq 1-\ve$. It is known (see, for instance
 \cite{Br}) that this property is equivalent to $\phi$-mixing with an exponentially fast decaying $\phi(n)$ coefficient. For 
 given sequences $\xi=(\xi_0,\xi_1,...)$ and $\eta=(\eta_0,\eta_1,...)$ with $\xi_i,\eta_i\in\cA$ we set $A^\xi_n=\{\om\in\cA^\bbN:\,
 \om_i=\xi_i,\, i=0,1,...,n-1\}$, $\tau_n^\eta(\om)=\min\{ k\geq 0:\,(X_k,X_{k+1},...,X_{k+n-1})=(\eta_0,\eta_1,...,\eta_{n-1})\}$ and
 $\Sig_{n,m}^{\xi,\eta}=\sum_{k=0}^{\tau_n^\eta-1}\bbI_{A^\xi_n}(X_k,X_{k+1},...,X_{k+n-1})$. Then, assuming that $\xi,\eta$ are nonperiodic
 and not shifts of each other we obtain from Corollary \ref{cor2.2} convergence in distribution of $\Sig_{n,m}^{\xi,\eta}$ to a geometric random
  variable with the parameter $(1+\la)^{-1}$ provided $\frac {\mu(A^\xi_n)}{\mu(A_n^\eta)}\to\la$ as $n\to\infty$.


\end{document}